\documentclass{IEEEtran}

\usepackage{amsmath,amssymb,amsthm}
\usepackage{algorithm}
\usepackage{algpseudocode}
\usepackage{array}
\usepackage[c2,nocomma]{optidef}
\usepackage[caption=false,font=normalsize,labelfont=sf,textfont=sf]{subfig}
\usepackage{textcomp}
\usepackage{stfloats}
\usepackage{url}
\usepackage{verbatim}
\usepackage{graphicx}
\usepackage{nicematrix}
\usepackage{tikz}
\usepackage{color}
\usepackage{bm}
\usepackage{booktabs}
\usepackage[colorlinks,linkcolor=blue,citecolor=blue]{hyperref}
\usepackage[capitalize,nameinlink]{cleveref}

\theoremstyle{plain}
\newtheorem{theorem}{Theorem}
\newtheorem{lemma}{Lemma}
\newtheorem{proposition}{Proposition}

\theoremstyle{definition}
\newtheorem{assumption}{Assumption}

\newtheorem{problem}{Problem}

\newtheorem{remark}{Remark}

\crefname{assumption}{Assumption}{Assumptions}

\newcommand{\arxiv}[2]{#1}
\newcommand{\esr}{\rho}

\def\BibTeX{{\rm B\kern-.05em{\sc i\kern-.025em b}\kern-.08em
    T\kern-.1667em\lower.7ex\hbox{E}\kern-.125emX}}

\begin{document}

\title{Adaptive Stepsizes With Certified Convergence in Distributed Gradient Tracking With Quadratic Costs}
\author{Yifan~Wang,~\IEEEmembership{Student~Member,~IEEE},
Luca~Ballotta,~\IEEEmembership{Member,~IEEE},
Ruggero~Carli,~\IEEEmembership{Member,~IEEE},
Andrea~Iannelli,~\IEEEmembership{Member,~IEEE},
Xianghui~Cao,~\IEEEmembership{Senior~Member,~IEEE},
and~Luca~Schenato,~\IEEEmembership{Fellow,~IEEE}
\thanks{This work was supported in part by the Frontier Technologies R\&D Program of Jiangsu Province of China Grant BF2024065, in part by the Shenzhen Science and Technology Program of China Grant JCYJ20230807114609019.}
\thanks{Y. Wang is with the School of Automation, Southeast University, Nanjing, 210096 China, and also with the
Department of Information Engineering, University of Padova, Padova 35131, Italy (e-mail: evan@seu.edu.cn).}
\thanks{X. Cao is with the School of Automation, Southeast University, Nanjing, 210096 China (e-mail:xhcao@seu.edu.cn).}
\thanks{A. Iannelli is with the Institute for System Theory and
Automatic Control, University of Stuttgart, Germany (e-mail: andrea.iannelli@ist.uni-stuttgart.de).}
\thanks{L. Ballotta, R. Carli, and L. Schenato are with 
the Department of Information Engineering, University of Padova, Padova 35131, Italy (e-mail: \{luca.ballotta, ruggero.carli, l.schenato\}@unipd.it).}}

\maketitle

\begin{abstract}
In this work, we propose an adaptive stepsize rule with guaranteed convergence for Distributed Gradient Tracking applied to scalar quadratic problems with heterogeneous curvatures. Most distributed gradient-based algorithms require a suitable stepsize selection. Available theoretical bounds are often overly conservative, while practical implementations typically rely on empirically tuned heuristics. Online adaptive strategies have only recently emerged for general distributed convex optimization, but their properties and performance remain only partially understood.
To gain analytical insight, we focus on the informative setting of scalar quadratic costs, which allows us to explicitly capture the interplay between network topology and curvature heterogeneity. We derive a convergence bound parameterized only by the essential spectral radius of the consensus matrix and the heterogeneity of the local cost curvatures, both computable online without any prior knowledge of the optimization problem. Optimizing this bound yields a computationally tractable surrogate for the convergence rate and the optimal constant stepsize. The resulting stepsize admits an analytical interpretation, guarantees convergence for arbitrary network topologies and curvature heterogeneity, and is provably tight for complete graphs and homogeneous curvatures. Finally, extensive numerical simulations demonstrate that the proposed distributed adaptive strategy significantly outperforms existing offline and online stepsize selection rules in the considered setting.

\end{abstract}

\section{Introduction}
\label{sec:introduction}
Distributed optimization has raised great attention in the past few decades due to its wide applicability, e.g., in multi-robot systems and sensor networks~\cite{TRO_2025_robot_backgrd,TAC_2024_sensor_net_backgrd}.
Standard distributed optimization algorithms include distributed (sub)gradient descent (DGD)~\cite{TAC_Nedic_2009,TAC_Nedic_2010}, distributed gradient tracking (DGT)~\cite{DiLorenzo2016NEXT,Qu_TCNS_2018,pu2021_Mathprogram_DGT}, and the alternating direction method of multipliers (ADMM)~\cite{Nicola_Robust_learning_TAC_2025}.
Specifically, the main drawback of DGD is that exact convergence is impossible unless a diminishing stepsize, which slows down convergence, is used. DGT achieves exact convergence with a sufficiently small stepsize using lightweight local updates~\cite{Ivano_TAC_2023_GT_is_integral,ST_GT_arxiv_2025}, but the design of the stepsize is often based on empirically tuned heuristics. ADMM guarantees robust exact convergence, but requires solving nested subproblems. 

Elaborating further, existing convergence analysis of DGT and stepsize design with formal guarantees of convergence are mostly conservative because they are based on global smoothness and strong convexity constants, i.e., worst-case curvature information~\cite{Qu_TCNS_2018,pu2021_Mathprogram_DGT,Ivano_TAC_2023_GT_is_integral,ST_GT_arxiv_2025}.
Indeed, empirical studies show that stepsizes based on Lipschitz constant estimation, either offline or online, yield slow convergence,
while heuristic tuning of larger stepsizes typically lacks formal guarantees. 
A related line of work on federated and distributed learning shows that second-order similarity provides structural information beyond worst-case curvature quantities.
For example, Network-DANE introduces a homogeneity parameter measuring the deviation of local Hessians from the global Hessian, and its convergence analysis shows that data homogeneity, network connectivity, and local averaging jointly affect the convergence rate~\cite{DANE_2020}.
Follow-up research further exploited Hessian homogeneity or heterogeneity to improve communication-computation 
complexity~\cite{AIST_2022_similarity,SunY_SIAM_2022_GT_surrogation,pmlr_2024_Drift_Correction,Nips_2023_Average_Second_order_Similarity,pmlr_2025_exploiting_similarity,TSP_2025_Effectiveness_Local_updates_data_heterogeneity}.
However, it remains unclear how curvature information, including Hessian homogeneity and heterogeneity, can be used to characterize the convergence of DGT dynamics beyond worst-case Lipschitz-based bounds. 

Another source of conservatism stems from the coupled consensus and gradient tracking dynamics of DGT which, by intertwining the local cost curvatures and the spectrum of the consensus matrix, jointly determine the stabilizing stepsizes. As a result, precisely estimating the speed of convergence of DGT for a given stepsize requires global information of network weights and local costs, which are hardly accessed by individual agents, not flexible to different optimization problems, and provide little intuition for principled design. Existing works establish convergence by deriving contractive error recursions through chains of norm inequalities whose coefficients depend on the network spectral parameter and worst-case curvature quantities, resulting in loose bounds~\cite{Compress_DGT_pushi_TAC_2022}.

\subsubsection*{Contribution}
Motivated by these challenges, we tackle the problem of stepsize selection in DGT targeting formal convergence certification and practically useful convergence speed.
Since the problem is challenging in the general case, we focus on scalar quadratic costs.
This setting is analytically tractable and allows us to explicitly study the coupling between network mixing and local curvature heterogeneity in the convergence of DGT, which we propose as a novel approach for stepsize design.
Our goal is to pursue a coordinated stepsize for fast convergence of DGT, using only succinct local information.
We derive a convergence certificate shaped by two problem-related scalar quantities: curvature dispersion $\eta$ (i.e., the maximum relative deviation of the local curvatures from the average) and the essential spectral radius $\sigma$ of the consensus matrix. Incidentally, these two quantities can be computed online via consensus tracking and max-tracking. Minimizing this certificate leads to a convergence-certified stepsize rule. In particular, the stepsize rule coincides with the optimal constant stepsize choice in both the complete-graph and uniform-curvature regimes for the worst-case root radius minimization problem using only $\sigma$. We further provide a distributed adaptive stepsize rule such that uncoordinated local stepsizes almost surely converge to this certified value.

\textit{Organization:}
Section~\uppercase\expandafter{\romannumeral2} introduces the quadratic optimization setting, formulates the constant-stepsize design problem for DGT, and motivates the proposed approach. Section~\uppercase\expandafter{\romannumeral3} derives the optimal constant stepsize in two analytically tractable regimes and establishes the basis for the convergence-certified design developed in Section~\uppercase\expandafter{\romannumeral4}. Section~\uppercase\expandafter{\romannumeral4} develops a convergence-certified constant-stepsize rule for the general case. Section~\uppercase\expandafter{\romannumeral5} presents a fully distributed adaptive implementation of the proposed design. Section~\uppercase\expandafter{\romannumeral6} reports numerical results, and Section~\uppercase\expandafter{\romannumeral7} concludes the paper.







\section{Setup and Problem Formulation}
\subsection{Quadratic Optimization Setting and Distributed Gradient Tracking Algorithm}
Consider a set of agents $\mathcal{V}=\{1,2,\dots,N\}$ interacting over a graph $\mathcal G=(\mathcal V, \mathcal E)$ with edges $\mathcal E\subseteq\{(i,j):i,j\in\mathcal V,\ i\neq j\}$. The set of neighbors of node $i$ is $\mathcal N_i := \{j\in\mathcal V : (i,j)\in\mathcal E\}$.
Communication weights are
collected in matrix $W=[w_{ij}]\in\mathbb R^{N\times N}$ with $w_{ij}\in(0,1)$ if $(i,j)\in\mathcal{E}$ or $i=j$ and $0$ otherwise.
\begin{assumption}\label{ass:graph_connected}
    Graph $\mathcal G$ is undirected and connected.
    Consensus matrix $W$ is symmetric doubly stochastic.
\end{assumption}
Under \cref{ass:graph_connected}, $W$ has real eigenvalues labeled as $1=\lambda_1>\lambda_2\ge\cdots\ge\lambda_N>-1$. The essential spectral radius of $W$ is defined as $\sigma:=\esr(W-\mathbf 1\mathbf 1^\top/N)=\max\{|\lambda_2|,|\lambda_N|\}<1$,
where $\esr(A)$ is the spectral radius of the matrix $A$.

\begin{assumption}\label{ass:cost_function}
    Each agent $i$ has access to the scalar quadratic cost $f_i(x)=\frac{1}{2}h_i(x-b_i)^2+c_i$ with $x, b_i, c_i\in\mathbb R$ and $h_i>0$.     
\end{assumption}
The goal is to collaboratively minimize the global cost, i.e.,
\begin{equation}
\min_{x\in\mathbb{R}} f(x) \qquad 
f(x):=\frac{1}{N}\sum_{i\in\mathcal V}  \left(\frac{1}{2}h_i(x-b_i)^2+c_i\right).
\label{eq:global_obj}
\end{equation}
A popular algorithm to solve~\eqref{eq:global_obj} is the following DGT. At each step $k$, each agent performs the following  updates
\begin{subequations}\label{eq:dgt_scalar}
\begin{align}
x_{i,k+1}
&=
\sum_{j\in\mathcal N_i\cup\{i\}} w_{ij}\bigl(x_{j,k}-\alpha_j y_{j,k}\bigr),
\label{eq:dgt_scalar_x}
\\
y_{i,k+1}
&=
\sum_{j\in\mathcal N_i\cup\{i\}} w_{ij}y_{j,k}
+
h_i\bigl(x_{i,k+1}-x_{i,k}\bigr).
\label{eq:dgt_scalar_y}
\end{align}
\end{subequations}
Here, $x_{i,k}\in\mathbb{R}$ is agent $i$'s local estimate of the optimal solution $x^\star$ of \eqref{eq:global_obj} and $y_{i,k}\in\mathbb{R}$ is its estimate of the gradient of the global cost, with $
y_{i,0}=h_i x_{i,0}$ and $x_{i,0}\in\mathbb R$.

We next introduce a few convenient notations.
Define the scaled gradient estimate $\tilde y_{i,k}=\alpha_i y_{i,k}$ and the stepsize matrix $\bm \alpha:=\operatorname{diag}(\alpha_1,\ldots,\alpha_N)$.
By stacking $H:=\operatorname{diag}(h_1,\ldots,h_N)$, $x_k:=\operatorname{col}(x_{1,k},\ldots,x_{N,k})$, 
and $\tilde y_k:=\operatorname{col}(\tilde{y}_{1,k},\ldots,\tilde{y}_{N,k})$,
the DGT updates~\eqref{eq:dgt_scalar} can be compactly written as the linear system
\begin{equation}
\hspace{-2pt}\begin{bmatrix}
x_{k+1}\\
\tilde y_{k+1}
\end{bmatrix}
=
\begin{bmatrix}
W & -W\\
\bm\alpha H(W-I) & \bm\alpha W \bm\alpha^{-1} - \bm\alpha H W
\end{bmatrix}
\begin{bmatrix}
x_k\\
\tilde y_k
\end{bmatrix}.
\label{eq:compact_dgt}
\end{equation}
We respectively define average curvature $\bar h$ and average-normalized curvature dispersion (parameter) $\eta$ as
\begin{equation}
    \bar h:=\frac{1}{N}\sum_{i\in\mathcal V}h_i \qquad
    \eta:=\max_{i\in\mathcal V}\frac{|h_i-\bar h|}{\bar h},
\end{equation}
where $\eta$ quantifies the heterogeneity of costs $f_i$ as the maximum relative deviation of local curvatures $h_i$ from the average.

\subsubsection*{High-level design objective}

Choosing the stepsizes $\alpha_i$ while formally ensuring convergence of DGT to the optimizer is nontrivial in practice.
Rather than designing the stepsize directly from the complete problem data, we seek an approach that not only certifies convergence but also yields practically useful convergence speed.
The system dynamics in~\eqref{eq:compact_dgt} jointly depend on the consensus matrix $W$ and the local curvature matrix $H$.
Even if $\alpha_i:=\alpha>0$, curvature heterogeneity prevents the system matrix in~\eqref{eq:compact_dgt} from being diagonalizable, making its eigenvalues unavailable in closed form.
As a result, any attempt to design $\alpha$ directly based on $W$ and $H$ is computationally cumbersome and provides little insight into principled stepsize design.
Instead, we seek a simpler characterization based only on two succinct problem descriptors: the essential spectral radius $\sigma$ and the curvature dispersion~$\eta$.

Informally, let $\rho_{\rm DGT}(\alpha;W,H)$ denote the asymptotic convergence factor of the DGT dynamics.
The ideal constant-stepsize design problem can then be formulated as
\begin{equation}
\!\!\!\alpha^\star(W,H)\!\in\!\arg\min_{\alpha>0}\,
\rho_{\rm DGT}(\alpha;W,H)\quad \mathrm{s.t.}\;
x_{i,k}\!\to\!x^\star,\ \forall i.
\label{eq:prob-high-level-nominal}
\end{equation}

Instead, we seek a computationally tractable surrogate depending only on $(\sigma,\eta)$:
\begin{equation}
\!\!\!\widehat\alpha^\star(\sigma,\eta)\!\in\!\arg\min_{\alpha>0}\,
\widehat\rho_{\rm DGT}(\alpha;\sigma,\eta)\quad \mathrm{s.t.}\;
x_{i,k}\!\to\!x^\star,\ \forall i,
\label{eq:prob-high-level-surrogate}
\end{equation}
with the property 
$\widehat{\rho}_{\rm DGT}(\alpha;\sigma,\eta)\ge
\rho_{\rm DGT}(\alpha;W,H)$
thus providing an upper bound depending only on $\sigma$ and $\eta$.
Although $\widehat\alpha^\star(\sigma,\eta)$ is generally suboptimal with respect to
$\alpha^\star(W,H)$,
it only requires the two scalar quantities $(\sigma,\eta)$,
thus providing analytical insight and enabling a practical distributed implementation.

Although the previous optimization problems assumes a common constant stepsize,
we will develop a distributed implementation whereby all agents asymptotically agree on this value, exploiting the fact that both $\sigma$ and $\eta$ can be computed distributively through consensus-like subroutines.
The main challenge is therefore to construct a suitable upper bound
$\widehat{\rho}_{\rm DGT}(\alpha;\sigma,\eta)$,
which we address in~\cref{sec:special-regimes,sec:heterogeneous}.
In the next subsection, we show that, for the quadratic setting considered here,
$\widehat\rho_{\rm DGT}(\alpha;W,H)$ coincides with the spectral radius of a suitable reduced system matrix.

\subsection{Problem Statement: Pursuing Fast DGT Convergence}
For the sake of readability,
we define the normalized local curvature
$d_i:=h_i/\bar h$, normalized curvatures matrix $D:=\operatorname{diag}(d_1,\ldots,d_N)$, and normalized stepsize $\beta:=\alpha\bar h$.
In the following, we consider $\beta$ instead of $\alpha$ to simplify notation.

Let $\mathcal T:=\left[v\quad U\right]\in \mathbb R^{N\times N}$ denote the orthogonal basis with $v={1}/{\sqrt{N}}\mathbf{1}$ and where $U$ contains the eigenvectors of $W$ associated with eigenvalues $\lambda_2,\ldots,\lambda_N$.
To split the consensus and disagreement dynamics, we define $\widehat x_{c,k}=v^\top x_k$, $\widehat x_{d,k}=U^\top x_k$, $\widehat y_{c,k}=v^\top\widetilde y_k$, and $\widehat y_{d,k}=U^\top\widetilde y_k$, and collect them into $\widehat x_k:=\operatorname{col}(\widehat x_{c,k},\widehat x_{d,k})$ and $\widehat y_k:=\operatorname{col}(\widehat y_{c,k},\widehat y_{d,k})$.
Applying the orthogonal change of basis induced by $\operatorname{diag}\{\mathcal T,\mathcal T\}$ to system~\eqref{eq:compact_dgt} results in the transformed reordered state
$[\widehat x_{c,k},\widehat y_{c,k},\widehat x_{d,k},\widehat y_{d,k}]^\top$
with system matrix $A_{\mathcal T}(\beta)$:
\begin{align}
    A_{\mathcal T}(\beta)=
    \begin{bmatrix}
        1 & \ast\\
        0 &  A(\beta)
    \end{bmatrix}
\end{align}
\begin{equation}
 \hspace{-7pt}A(\beta)
\hspace{-1.5pt}=\hspace{-1.5pt}
\begin{bmatrix}
1-\beta & \beta\delta^\top(\Lambda-I) & -\beta\delta^\top\Lambda\\
0 & \Lambda & -\Lambda\\
-\beta\delta & \beta(I+\widetilde{\Delta})(\Lambda-I)
& \bigl(I-\beta(I+\widetilde{\Delta})\bigr)\Lambda
\end{bmatrix}\hspace{-1pt}
\label{eq:Ahat}
\end{equation}
where the raw block $\ast$ does not affect the spectrum of $A_{\mathcal T}(\beta)$, $\Lambda:=\operatorname{diag}(\lambda_2,\ldots,\lambda_N)$, $\delta:=U^{\top} \Delta v$, $\widetilde\Delta:=U^{\top} \Delta U$ and $\Delta:=D-I$.
We name $A(\beta)$ the reduced system matrix; its spectral radius denoted by $\esr(A(\beta))$ rules the convergence speed of \eqref{eq:compact_dgt} and $x_k\to \mathbf 1 x^\star$ iff $\esr(A(\beta))<1$.
Since $A(\beta)$ is affine in $\beta$ and $\esr(A(\beta))$ is
continuous in the matrix entries, the following standard
lemma holds, which is instrumental to the well-posedness of the ensuing stepsize selection problem. 
\begin{lemma}
\label{lem:ess_radius_continuity}
The map
$\beta\mapsto \esr(A(\beta))$ is continuous.
\arxiv{}{\begin{IEEEproof}
The system matrix in \eqref{eq:compact_dgt} has a structural unit mode
that is independent of $\beta$. After separating this mode, the
remaining nontrivial block has entries that depend continuously on
$\beta$. Since the spectral radius is continuous w.r.t. its entries of a finite-dimensional matrix, the spectral radius $\esr(A(\beta))$ is continuous in $\beta$. 
\end{IEEEproof}}
\end{lemma}

We are now ready to instantiate the nominal optimization problem~\eqref{eq:prob-high-level-nominal}.
Since the normalized stepsize is defined as $\beta=\alpha\bar h$, optimizing over $\alpha$ is equivalent to optimizing over $\beta$, with $\alpha^\star=\beta^\star/\bar h$.

\begin{problem}
Find the normalized stepsize $\beta$ minimizing the asymptotic convergence factor of DGT:
\begin{equation}
\beta^\star=\beta^\star(W,H)
\in
\arg\min_{\beta>0}
\rho(A(\beta)).
\label{eq:ideal_beta_problem}
\end{equation}
\end{problem}

Note that
$$\rho_{\rm DGT}(\alpha;W,H)=\rho_{\rm DGT}(\beta/\bar h;W,H)= \rho(A(\beta)),$$ therefore the optimization problems~\eqref{eq:ideal_beta_problem} and \eqref{eq:prob-high-level-nominal} are equivalent.  

Before tackling the problem in its full generality, in the next~\cref{sec:special-regimes} we first analyze two special regimes in which problem~\eqref{eq:ideal_beta_problem} can be solved exactly using only $\sigma$ and $\eta$. These results establish the theoretical basis for the convergence-certified design developed in~\cref{sec:heterogeneous}.

\section{Optimal Constant Stepsize in Special Regimes}
\label{sec:special-regimes}

In this section, we focus on two cases solvable in closed form, i.e., complete graph and uniform curvatures, respectively.

\subsection{Complete Graph: \texorpdfstring{$\sigma=0,\eta\ge0$}{sigma=0,eta>=0}}\label{sec:complete_graph}
In this case, the consensus matrix reduces to $W=\frac{1}{N}\mathbf 1 \mathbf 1^\top$ and 
$\Lambda=0$. The reduced system matrix $ A(\beta)$ becomes
\begin{equation}\label{eq:Ahat_complete}
     A(\beta)=
    \begin{bmatrix}
    1-\beta & -\beta\delta^\top & 0\\
    0 & 0 & 0\\
    -\beta\delta & -\beta(I+\widetilde{\Delta}) & 0
    \end{bmatrix}.
\end{equation}
Although $A(\beta)$ in~\eqref{eq:Ahat_complete} depends on heterogeneity-dependent terms $\delta$ and $\widetilde{\Delta}$, they do not appear in
its characteristic polynomial. Direct calculation yields $\det\bigl(zI- A(\beta)\bigr)=z^{2(N-1)}\bigl(z-(1-\beta)\bigr)$.
Hence, the spectral radius is $\esr(A(\beta))=|1-\beta|$ and is minimized at the normalized stepsize 
\begin{equation}
    \beta^\star=1.
\end{equation}

\subsection{Uniform Curvatures and dominant $\lambda_2$: $\eta=0, \lambda_2\!\geq\!|\lambda_N|$\label{sec:homo}}
In this case, $D=I$, $\delta=0$, and $\widetilde\Delta=0$. We denote the corresponding \emph{homogeneous} reduced matrix $A(\beta)$ by $A_0(\beta)$:
\begin{equation}\label{eq:A_0}
     A_0(\beta)
    =
    \begin{bmatrix}
    1-\beta & 0 & 0\\
    0 & \Lambda & -\Lambda\\
    0 & \beta(\Lambda-I) & (1-\beta)\Lambda
    \end{bmatrix}.
\end{equation}
The characteristic polynomial of $ A_0(\beta)$ factorizes as $\det(zI- A_0(\beta))=(z-1+\beta)\prod_{i=2}^{N}P_{\lambda_i}(z;\beta)$ where $P_{\lambda_i}(z;\beta):=\det(zI-B_i(\beta))=(z-\lambda_i)^2+\beta\lambda_i(z-1)$ and each block $B_i(\beta)$ is associated with an eigenvalue $\lambda_i<1$:
\begin{equation}
    B_i(\beta):=
    \begin{bmatrix}
    \lambda_i & -\lambda_i\\
    \beta(\lambda_i-1) & (1-\beta)\lambda_i
    \end{bmatrix}.\label{eq:block_Bi}
\end{equation}
In this diagonalized scenario, $\beta^\star$ potentially depends on all eigenvalues of $W$. However, common consensus weight constructions often satisfy $\lambda_2\geq|\lambda_N|$, and hence $\sigma=\lambda_2$
\cite{Boyd_SIAMR_2004,Oreshkin_TSP_2010}, in which case $\beta^\star$ is shown to depend on $\sigma$. This allows us to compute the optimal $\beta^\star$ in closed form. In fact, consider the generic polynomial $P_\lambda(z;\beta):=(z-\lambda)^2+\beta\lambda(z-1)$.
For $\beta>0$, define the worst-case root radius as
\begin{equation}\label{eq:R_0_define}
    R_0(\beta,\sigma):=\max\Big\{|1-\beta|,
    \sup_{\lambda\in[-\sigma,\sigma]}\max{|z|:P_\lambda(z;\beta)=0}
    \Big\}.
\end{equation}
By construction, in total generality $\rho(A_0(\beta))\le R_0(\beta,\sigma)$, however, in dominating second eigenvalue when $\lambda_2>|\lambda_N|$, we have
\begin{equation}\label{eq:stepsize_problem}
   \beta^\star=\arg\min_{\beta>0}\,\,\,R_0(\beta,\sigma)=1-\sqrt{\sigma}
\end{equation}
as proved in the following lemma. 
\begin{lemma}
\label{lem:homogeneous_envelope}
Under the conditions $\sigma\in(0,1)$, $\eta=0$ and $\lambda_2\geq |\lambda_N|$, problem~\eqref{eq:ideal_beta_problem} has the unique solution $\beta^\star=1-\sqrt{\sigma}$, with
$\rho(A_0(\beta^\star))=R_0(\beta^\star,\sigma)=\sqrt{\sigma}$.
\end{lemma}
\begin{IEEEproof}
    For fixed $\beta\in(0,2)$, the root radius of $P_\lambda(z;\beta)$ is
    decreasing in $\lambda$ on $[-\sigma,0]$ and increasing on
    $[0,\sigma]$. Moreover, $\lambda_2\geq|\lambda_N|$ implies
    $\sigma=\lambda_2$, so $\lambda=\sigma$ is an actual eigenmode and
    $\rho(A_0(\beta))\geq\max\{|1-\beta|,R_+(\beta,\sigma)\}$, where
    \begin{equation}
    R_+(\beta,\sigma)
    :=\frac{\sigma(2-\beta)
    +\sqrt{\beta\sigma\left(4(1-\sigma)+\beta\sigma\right)}}{2}.
    \end{equation}
    
    For $0<\beta\leq1$, $1-\beta$ is strictly decreasing, whereas
    $R_+(\beta,\sigma)$ is strictly increasing. Their unique intersection
    follows from
    $P_\sigma(1-\beta;\beta)
    =(1-\sigma)(\beta-1+\sqrt{\sigma})(\beta-1-\sqrt{\sigma})$
    and is given by $\beta^\star=1-\sqrt{\sigma}$. At the negative
    endpoint $\lambda=-\sigma$, the roots form a complex-conjugate pair
    with modulus
    \begin{equation}
        R_-(\beta,\sigma):=\sqrt{\sigma(\sigma+\beta)}.
    \end{equation}
    At $\beta^\star$,
    $1-\beta^\star=R_+(\beta^\star,\sigma)=\sqrt{\sigma}$, whereas
    $R_-(\beta^\star,\sigma)<\sqrt{\sigma}$; hence,
    $R_0(\beta^\star,\sigma)=\sqrt{\sigma}$. Since
    $\lambda_2=\sigma$ is an actual eigenvalue,
    $\rho(A_0(\beta^\star))\geq\sqrt{\sigma}$, while
    $\rho(A_0(\beta^\star))\leq R_0(\beta^\star,\sigma)$. Therefore,
    $\rho(A_0(\beta^\star))=\sqrt{\sigma}$. For every other
    $\beta\in(0,1]$, it has 
    $\rho(A_0(\beta))\geq\max\{1-\beta,R_+(\beta,\sigma)\}
    >\sqrt{\sigma}$.
    
    For $1<\beta<2$,
    $P_\sigma(\sqrt{\sigma};\beta)
    =\sigma(1-\sqrt{\sigma})(1-\sqrt{\sigma}-\beta)<0$; hence, the
    larger root of $P_\sigma$ exceeds $\sqrt{\sigma}$ and so does
    $\rho(A_0(\beta))$. Finally, for $\beta\geq2$, one has 
    $\rho(A_0(\beta))\geq|1-\beta|\geq1>\sqrt{\sigma}$. Therefore, problem \eqref{eq:ideal_beta_problem} has the unique solution
    $\beta^\star=1-\sqrt{\sigma}$, with
    $\rho(A_0(\beta^\star))=R_0(\beta^\star,\sigma)=\sqrt{\sigma}$.
\end{IEEEproof}

Note that for $\sigma=0$, the result in \cref{lem:homogeneous_envelope} matches the complete-graph optimum derived in \cref{sec:complete_graph}. Moreover, the root-locus analysis in its proof gives an explicit formula for the worst-case root radius defined in \eqref{eq:R_0_define} as follows:
\begin{equation}\label{eq:root_envelop}
    \begin{aligned}
        R_0(\beta,\sigma)
        &:=\max\{|1-\beta|,R_T(\beta,\sigma)\}
    \end{aligned}
\end{equation}
where $R_T(\beta,\sigma):=\max\{R_+(\beta,\sigma),R_-(\beta,\sigma)\}$
is increasing in $\beta\in(0,2)$ and satisfies
$\lim_{\beta\to0^+}R_T(\beta,\sigma)=\sigma$. The worst-case root radius $R_0(\beta,\sigma)$ and its components are illustrated in \cref{fig:root_envelope}. It serves as the nominal reference for the perturbation-based upper bound on $\esr(A(\beta))$ in the next section.


\section{Stabilizing Constant Stepsize in General Case}\label{sec:heterogeneous}
In this section, we consider DGT in the general regime $\sigma\in(0,1)$ and $\eta>0$, namely, non-complete graph and non-uniform (i.e., heterogeneous) curvatures. Unlike uniform curvatures, the eigenbasis of $W$ does not diagonalize the curvatures matrix $D$. In $A(\beta)$, this appears through $\delta$ and
$\widetilde\Delta$: the terms involving $\delta$ link the consensus component $\widehat y_{c,k}$ of the gradient tracker with the disagreement variables $(\widehat x_{d,k},\widehat y_{d,k})$, while the off-diagonal entries of $\widetilde\Delta$ mix the disagreement modes associated with
different eigenvalues $\lambda_2,\ldots,\lambda_N$. As such, unlike in \cref{sec:homo}, the characteristic polynomial of $A(\beta)$ no longer decomposes into the modal factors $P_{\lambda_i}(z;\beta)$. For this reason, in \cref{sec:worst-case-stability-certificate}, we interpret $A(\beta)$ as a perturbation of the homogeneous reduced matrix $A_0(\beta)$ and seek an upper bound $r<1$ such that $\esr(A(\beta))<r$. 
We refer to any such an $r$ as a \emph{certified radius}.
In \cref{subsec:heteroge_stepsize_rule}, we exploit this certified radius to numerically compute a stepsize $\beta$ that minimizes it for fastest convergence.

\subsection{Convergence Certification}
\label{sec:worst-case-stability-certificate}
\begin{figure}
    \centering
    \includegraphics[width=0.4\textwidth]{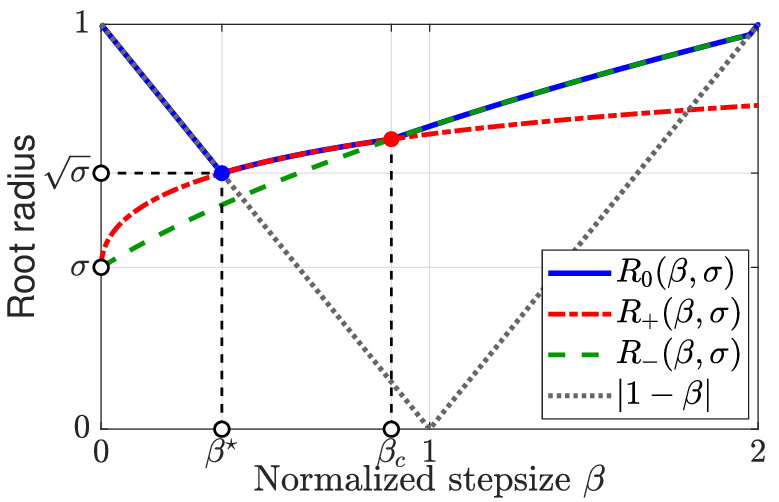}
    \caption{Worst-case root radius $R_0(\beta,\sigma)$ and its components under $\sigma=0.4$. $\beta^\star=1-\sqrt{\sigma}$ is the value when $1-\beta=R_T(\beta,\sigma)$. $\beta_{\rm c}\approx0.8835$ is the nonzero value when $R_+(\beta,\sigma)=R_-(\beta,\sigma)$.}
    \label{fig:root_envelope}
\end{figure}

We first derive a certified upper bound on $\esr(A(\beta))$. The modal factorization of the characteristic polynomial in \cref{sec:homo}, although unavailable for $A(\beta)$ under heterogeneous curvatures, provides a nominal reference for the perturbation argument. Specifically, we decompose $A(\beta)=A_0(\beta)+A_\Delta(\beta)$ where the heterogeneity-induced perturbation part is
\begin{equation}
     A_\Delta(\beta)=
    \begin{bmatrix}
    0 & \beta\delta^\top(\Lambda-I) & -\beta\delta^\top\Lambda\\
    0 & 0 & 0\\
    -\beta\delta & \beta\widetilde\Delta(\Lambda-I) & -\beta\widetilde\Delta\Lambda
    \end{bmatrix}.
\end{equation}
 We bound the effect of $A_\Delta(\beta)$ on the homogeneous reduced matrix $A_0(\beta)$ via perturbation analysis, leading to the key convergence certificate for the principled stepsize design.

\begin{lemma}[Convergence certificate]\label{lem:stability_certificate}
    Let $\sigma\in(0,1)$, $\eta>0$, $\beta\in(0,2)$, and
    $r\in(R_0(\beta,\sigma),1)$ be given.
    If $\Phi(r,\beta,\sigma,\eta)<1$, then $\esr(A(\beta))<r$, where
    \begin{equation}\label{eq:phi}
    \hspace{-4pt}\Phi(r,\beta,\sigma,\eta)=\beta\left( \eta+\frac{\beta\eta^2}{r-|1-\beta|}\right)\frac{\sigma(1+r)}{\left(r-R_T(\beta,\sigma)\right)^2}.
    \end{equation}
\end{lemma}

\begin{IEEEproof}
Consider the homotopy $A_\tau(\beta):=A_0(\beta)+\tau A_\Delta(\beta)$, $\tau\in[0,1]$, with $A_1(\beta):=A(\beta)$. We show that no eigenvalue of $A_\tau(\beta)$ crosses $|z|=r$ for any $\tau\in[0,1]$.

For any $z$ on the contour $|z|=r$, since $r>R_0(\beta,\sigma)\ge |1-\beta|$,
one has $z-|1-\beta|\neq0$. Applying the Schur complement to the scalar block
$z-|1-\beta|$ in $zI-A_\tau(\beta)$ gives
\begin{equation}
    \det(zI-A_\tau(\beta))=
    (z-|1-\beta|)\det S_\tau(z),
    \label{eq:det_schur_tau}
\end{equation}
where $S_\tau(z)=S_0(z)+\beta L F_{\Delta,\tau}(z)R$ with
\begin{equation}
    S_0(z)=
    \begin{bmatrix}
    zI-\Lambda & \Lambda\\
    -\beta(\Lambda-I) & zI-\Lambda+\beta\Lambda
    \end{bmatrix},
\end{equation}
\begin{equation}
    L=
    \begin{bmatrix}
    0 & I
    \end{bmatrix}^{\top},
    \qquad
    R=
    \begin{bmatrix}
    -(\Lambda-I) & \Lambda
    \end{bmatrix},
\end{equation}
and
\begin{equation}
    F_{\Delta,\tau}(z)=\tau\widetilde \Delta-\frac{\beta\tau^2}{z-|1-\beta|}\delta\delta^\top.
    \label{eq:F_Delta_tau_def}
\end{equation}
Since $r>R_0(\beta,\sigma)$, \cref{lem:homogeneous_envelope}
implies that $\det S_0(z)\neq0$ for all $z$ on $|z|=r$. Therefore, using the
matrix determinant lemma, one obtains
\begin{equation}\label{eq:det_S_tau_factor}
    \det S_\tau(z)=
    \det S_0(z)
    \det(I\hspace{-0.5pt}+\hspace{-0.5pt}\beta F_{\Delta,\tau}(z)\underbrace{R S_0(z)^{-1}L}_{:=Q(z)}).
\end{equation}
It remains to bound the perturbation determinant factor $\det(I+\beta F_{\Delta,\tau}(z)Q(z))$ in \eqref{eq:det_S_tau_factor}. For any $z$ on the contour $|z|=r$, since $0<\beta<2$ and $r>R_0(\beta,\sigma)\ge |1-\beta|$, the reverse triangle inequality gives $|z-|1-\beta||\ge\bigl||z|-|1-\beta|\bigr|=r-|1-\beta|$. Therefore, one has
\begin{align}
    \|F_{\!\Delta,\tau}(z)\|_2\!&\!\le
    \tau\|\widetilde\Delta\|_2\!+\!\frac{\beta\tau^2}{|z\!-\!|1\!-\!\beta||}
    \|\delta\delta^\top\|_2 
    \overset{(i)}{\le}\!\eta\!+\!\frac{\beta\eta^2}{r\!-\!|1\!-\!\beta|},
    \label{eq:F_bound_coarse}
\end{align}
$\forall\tau\in[0,1]$ where $(i)$ follows from $\|\widetilde \Delta\|_2\leq \|\Delta\|_2$ and $\|\delta\|_2\leq \|\Delta\|_2$.


On the other hand, since $\Lambda$ is diagonal, $S_0(z)$ decomposes
into $N-1$ independent blocks given in \eqref{eq:block_Bi}. Therewith, $Q(z)$ is diagonal with entries $q_i(z)=\frac{\lambda_i(z-1)}{(z-\lambda_i)^2+\beta\lambda_i(z-1)}$ for $i=2,\ldots,N$. The denominator of $q_i(z)$ is exactly the polynomial $P_{\lambda_i}(z;\beta)$. 
By the definition of $R_T(\beta,\sigma)$, the two roots of $P_{\lambda_i}(z;\beta)$ are contained in the disk $|z|\le R_T(\beta,\sigma)$. Hence, for $|z|=r$ with $r>R_T(\beta,\sigma)$, one obtains $|P_{\lambda_i}(z;\beta)|\ge\bigl(r-R_T(\beta,\sigma)\bigr)^2$. The numerator of $q_i(z)$ is upper bounded by $|\lambda_i(z-1)|\le\sigma(1+r)$ on $|z|=r$ because $|\lambda_i|\leq \sigma$. Plugging these relations into $\|Q(z)\|_2$ results in
\begin{equation}
    \|Q(z)\|_2=\max_{i=2,\ldots,N}|q_i(z)|
    \le
    \frac{\sigma(1+r)}
    {(r-R_T(\beta,\sigma))^2}.
    \label{eq:T0_bound_coarse}
\end{equation}
Finally, combining \eqref{eq:F_bound_coarse} and \eqref{eq:T0_bound_coarse}
yields
\begin{equation}
    \beta\|F_{\Delta,\tau}(z)\|_2\|Q(z)\|_2
    \le\Phi(r,\beta,\sigma,\eta)<1.
\end{equation}
Hence $I+\beta F_{\Delta,\tau}(z)Q(z)$ is non-singular for all $|z|=r$ and all $\tau\in[0,1]$. By \eqref{eq:det_schur_tau} and \eqref{eq:det_S_tau_factor}, this implies $\det(zI-A_\tau(\beta))\neq0$ for all $z$ on $|z|=r$ and all $\tau\in[0,1]$. Since $A_\tau(\beta)$ depends continuously on $\tau$, the number of
eigenvalues inside $|z|<r$ is invariant along the homotopy. At
$\tau=0$, all eigenvalues of the homogeneous nominal matrix
$A_0(\beta)$ lie inside $|z|<r$ because $r>R_0(\beta,\sigma)$. Hence
the same holds for $A_1(\beta)=A(\beta)$, and thus $\esr(A(\beta))<r$. The proof is completed.
\end{IEEEproof}

We are ready to state our main result.
A nontrivial certified radius can always be found;
the condition $\Phi(r,\beta,\sigma,\eta)<1$ is well posed, enabling numerical optimization of $\beta$.

\begin{theorem}[Feasibility of the convergence certificate]
\label{thm:N_node_certificate}
Given the essential spectral radius $\sigma\in(0,1)$ and the curvature dispersion $\eta>0$, there exist a normalized stepsize $\beta\in(0,2)$ and a certified radius $r\in(R_0(\beta,\sigma),1)$ such that $\esr(A(\beta))<r$.
\end{theorem}
\begin{IEEEproof}
  Given fixed $\sigma\in(0,1)$ and $\eta>0$, we seek such a stepsize for sufficiently small $\beta\in(0,1)$. For each fixed $\beta\in(0,1)$, the function $r\mapsto\Phi(r,\beta,\sigma,\eta)$ is continuous and strictly decreasing on $(R_0(\beta,\sigma),1)$. Define $g(\beta):=\lim_{r\to1^-}\Phi(r,\beta,\sigma,\eta)$. Since $|1-\beta|=1-\beta$ on $(0,1)$, direct calculation gives $g(\beta)=2\sigma\beta\eta(1+\eta)/(1-R_T(\beta,\sigma))^2$. As $R_T(\beta,\sigma)\to\sigma<1$ when $\beta\to0^+$, we have $g(\beta)\to0$. Moreover, $R_0(\beta,\sigma)<1$ for all sufficiently small $\beta>0$. Hence one can fix $\bar\beta\in(0,1)\subset(0,2)$ such that $g(\bar\beta)<1$ and $R_0(\bar\beta,\sigma)<1$. By continuity in $r$, there exists $\bar r\in(R_0(\bar\beta,\sigma),1)$ such that $\Phi(\bar r,\bar\beta,\sigma,\eta)<1$. Then, the theorem statement follows from~\cref{lem:stability_certificate}.
\end{IEEEproof}

\subsection{Certified Radius Optimization and Stepsize Selection}\label{subsec:heteroge_stepsize_rule}
Given any $(\sigma,\eta)\in[0,1)\times[0,\infty)$ and any $\beta\in(0,2)$, we define the minimal certified radius based on \cref{lem:stability_certificate} as
\begin{equation}\label{eq:r_cert}
    \hspace{-7pt}\widehat\esr_{\rm DGT}(\beta,\sigma,\eta)\hspace{-1pt}:=\hspace{-1pt}\inf\left\{
    r\hspace{-1pt}\in\hspace{-1pt}(R_0(\beta,\sigma),1)\hspace{-1pt}:\hspace{-1pt}\Phi(r,\beta,\sigma,\eta)\hspace{-1pt}<\hspace{-1pt}1\right\}
\end{equation}
with the convention that $\widehat\esr_{\rm DGT}(\beta,\sigma,\eta)=1$ if the right-hand side of~\eqref{eq:r_cert} is the empty set. Whenever $\widehat\esr_{\rm DGT}(\beta,\sigma,\eta)<1$, \cref{lem:stability_certificate} ensures $\esr(A(\beta))\leq\widehat\esr_{\rm DGT}(\beta,\sigma,\eta)<1$. Since the quantity $\widehat\esr_{\rm DGT}(\beta,\sigma,\eta)$ corresponds to the smallest radius $r$ certified by \cref{thm:N_node_certificate}, we consider the following generalization of~\eqref{eq:stepsize_problem} to the heterogeneous-curvature setting, which acts as a computationally tractable proxy of \eqref{eq:ideal_beta_problem}:
\begin{equation}\label{eq:cert_beta_arg_min}
    \widehat\beta^\star(\sigma,\eta)\in\operatorname*{arg\,min}_{0<\beta<2}\;
    \widehat\esr_{\rm DGT}(\beta,\sigma,\eta).
\end{equation}
In the sequel, we refer to $\widehat\beta^\star (\sigma,\eta)$ as certified stepsize and denote $\widehat\esr_{\rm DGT}^\star(\sigma,\eta):=\widehat\esr_{\rm DGT}(\widehat\beta^\star(\sigma,\eta),\sigma,\eta)$.
The following lemma establishes that the solution of~\eqref{eq:cert_beta_arg_min} yields a nontrivial bound and is continuous in $\sigma$ and $\eta$. These properties are foundational to our stepsize design.
\begin{lemma}[Continuity of the certified stepsize rule]
\label{lem:cert_stepsize_unique_continuous}
For any $(\sigma,\eta)\in[0,1)\times[0,\infty)$, the certified radius
minimization problem in \eqref{eq:cert_beta_arg_min} admits a unique solution $\widehat\beta^\star(\sigma,\eta)$. Moreover, the map $(\sigma,\eta)\mapsto\widehat\beta^\star(\sigma,\eta)$ is continuous on $[0,1)\times[0,\infty)$ and the minimal certified radius
satisfies $\widehat\esr_{\rm DGT}^\star(\sigma,\eta)<1$.
\end{lemma}
\begin{IEEEproof}
     The proof is provided in Appendix~\ref{appendix:A}.
\end{IEEEproof}

Given $(\sigma,\eta)\in[0,1)\times[0,\infty)$ and a candidate $\beta$, the function $\Phi(r,\beta,\sigma,\eta)$ is non-increasing in $r$ on $(R_0(\beta,\sigma),1)$. Hence, feasibility of $\beta$ can be checked by testing whether $R_0(\beta,\sigma)<1$ and $\Phi(1-\epsilon,\beta,\sigma,\eta)<1$ for some small $\epsilon>0$. If feasible, $\widehat\esr_{\rm DGT}(\beta,\sigma,\eta)$ can be obtained by a scalar bisection search on $r\in(R_0(\beta,\sigma),1)$. The outer minimization in \eqref{eq:cert_beta_arg_min} then reduces to a one-dimensional search on $\beta\in(0,2)$.

\begin{remark}[Recovery of special cases]
    The certified optimization~\eqref{eq:cert_beta_arg_min} recovers the optimal stepsize derived in \cref{sec:complete_graph,sec:homo}.
    If $\eta=0$ and $\lambda_2\geq |\lambda_N|$ (uniform curvatures and dominant second eigenvalue), then $\widehat\esr_{\rm DGT}(\beta,\sigma,0)=R_0(\beta,\sigma)=\esr(A(\beta))$ and~\eqref{eq:cert_beta_arg_min} reduces to worst-case root radius minimization yielding $\widehat\beta^\star(\sigma,0)=\beta^\star=1-\sqrt{\sigma}$. If $\sigma=0$ (complete graph), then $\widehat\esr_{\rm DGT}(\beta,0,\eta)=R_0(\beta,0)=|1-\beta|$ for $\beta\in(0,2)$, recovering the optimum $\beta^\star=\widehat\beta^\star(0,\eta)=1$.
\end{remark}

\arxiv{}{%
\subsection{Explicit Asymptotic Expansions}
In this subsection, we derive explicit asymptotic expansions of the proposed stepsize rule in \eqref{eq:cert_beta_arg_min} in the limits $\eta\to 0^+$ and $\eta\to+\infty$. These expansions provide closed-form approximations for both the certified stepsize and the associated certified radius. Note that the complete-graph case $\sigma=0$ is degenerate and gives $\beta^\star(0,\eta)=1$, we focus on the nondegenerate case $\sigma\in(0,1)$. For brevity, denote the certified stepsize and the corresponding certified radius with certain $\eta$ by $\beta_\eta:=\beta^\star(\sigma,\eta)$ and $r_\eta:=r_{\rm cert}(\beta_\eta,\sigma,\eta)$, respectively.

If the curvature-dispersion is in the weak limit $\eta\to 0^+$, we have the following expansions.
\begin{proposition}[Small curvature-dispersion regime]
\label{prop:weak_heterogeneity_expansion}
For $\sigma\in(0,1)$, the certified stepsize and the corresponding certified radius admit the following asymptotic expansions as $\eta\to0^+$:
\begin{align}
        \beta_\eta&=1-\sqrt{\sigma}-\frac{\sqrt{C_0}}{1+\kappa}\sqrt{\eta}+o(\sqrt{\eta}),\label{eq:beta_eta0}\\
        r_\eta&=\sqrt{\sigma}+\frac{\sqrt{C_0}}{1+\kappa}\sqrt{\eta}+o(\sqrt{\eta}),\label{eq:r_eta0} 
\end{align}
where $C_0=\sigma(1-\sigma)$ and $\kappa=\frac{\sqrt{\sigma}}{\sqrt{\sigma}+2}$.
\end{proposition}

\begin{IEEEproof}[Sketch of proof]
The proof views the case $\eta>0$ as a perturbation of the homogeneous optimizer $(\beta_0,r_0)=(1-\sqrt{\sigma},\sqrt{\sigma})$. By introducing two local gap variables, the certified-radius problem is reduced to a leading-order balance between the heterogeneity term and the active root-envelope gap. This balance determines the $\sqrt{\eta}$-scale correction, and the local linearization around $(\beta_0,r_0)$ gives the stated expansions.
Please see Appendix~\ref{appendix:B} for detailed calculations.
\end{IEEEproof}

In contrast, if the curvature-dispersion is in the strong limit: $\eta\to +\infty$, we have the following expansions.
\begin{proposition}[Large curvature-dispersion regime]\label{prop:strong_hete_expansion}
For $\sigma\in(0,1)$, the certified stepsize and the corresponding certified radius admit the following asymptotic expansions as $\eta\to+\infty$:
\begin{align}
        \beta_\eta&=\frac{(1-\sigma)^2}{4\sigma}\eta^{-2}+o(\eta^{-2}),\label{eq:beta_infty_eta}\\
        r_\eta&=1-\frac{(1-\sigma)^2}{8\sigma}\eta^{-2}+o(\eta^{-2}).\label{eq:r_infty_eta}
\end{align}
\end{proposition}

\begin{IEEEproof}[Sketch of proof]
We rewrite the certified radius relative to the admissible boundary as $r_\eta=1-\beta_\eta+g_\eta$ and derives the leading-order certificate in the strong heterogeneity regime. The resulting reduced feasibility problem determines the largest admissible stability margin $1-r_\eta$. Optimizing this limiting margin gives both the $\eta^{-2}$ scaling and the coefficients in the stated expansions. The detailed derivations are reported in Appendix~\ref{appendix:C}.
\end{IEEEproof}}

\section{Stabilizing Adaptive Uncoordinated Stepsizes}
\label{sec:distributed}

We now construct a fully distributed procedure that enables each agent to recover our approximate solution to problem~\eqref{eq:ideal_beta_problem}, i.e., $\alpha_i\to\widehat\alpha^\star(\sigma,\eta):=\widehat\beta^\star(\sigma,\eta)/\bar{h}$, using locally available information.
To circumvent dependence on global parameters, namely average curvature $\bar h$, essential spectral radius $\sigma$, and curvature dispersion $\eta$, we add a distributed estimation routine which runs in parallel to DGT updates.
Each agent $i$ performs:
\begin{align}
    \widehat {\bar h}_{i,k+1}&=
    \sum\nolimits_{j\in\mathcal N_i\cup\{i\}} w_{ij}\widehat{\bar h}_{j,k},\qquad \widehat{\bar h}_{i,0}=h_i\label{eq:ave_curv}\\
    \nu_{i,k+1}&=\sum\nolimits_{j\in\mathcal N_i\cup\{i\}} w_{ij}\nu_{j,k},\qquad \nu_{i,0}\stackrel{\rm i.i.d.}{\sim}\mathcal D_\nu\label{eq:ave_rand}
\end{align}
where $\mathcal D_\nu$ is an absolutely continuous distribution on $\mathbb R$ w.r.t. the Lebesgue measure, e.g., $\mathcal D_\nu=\mathcal N(0,1)$.
The update \eqref{eq:ave_curv} tracks the average curvature $\bar h$,
whereas \eqref{eq:ave_rand} is used to construct a power-iteration-based estimator of $\sigma$~\cite{saad2011numerical}. Specifically,
\begin{subequations}\label{eq:sigma_estimate}
    \begin{align}
    \widehat{\sigma}_{i,k}^{0}
    &=
    \sqrt{
   {|\nu_{i,k+1}-\nu_{i,k}|}/
         {|\nu_{i,k-1}-\nu_{i,k-2}|}},\label{eq:sigma_local_ratio}\\
    \widehat{\sigma}_{i,k}^{\ell}
    &=
    \max_{j\in\mathcal N_i\cup\{i\}}
    \left\{
    \widehat{\sigma}_{j,k-1}^{\ell-1}
    \right\},
    \qquad
    \ell=1,2,\ldots,m,\label{eq:sigma_max_consensus}
    \end{align}
\end{subequations}
where, in \eqref{eq:sigma_local_ratio}, we adopt a two-step ratio construction, avoiding parity oscillations that may arise when both $\sigma$ and $-\sigma$ are dominant eigenvalues of $W$. Let $Wu_j=\lambda_j u_j$ and
$p_j:=(\lambda_j-1)u_j^\top\nu_0$ where $(\lambda_j,u_j)$ denotes an orthonormal eigenpair of $W$, one has
$\nu_{i,k+1}-\nu_{i,k}
=\sigma^k(\sum_{\lambda_j=\sigma}p_j[u_j]_i+(-1)^k \sum_{\lambda_j=-\sigma}p_j[u_j]_i)+o(\sigma^k)$.
Since all dominant modal contributions
$p_j[u_j]_i$, with $|\lambda_j|=\sigma$, may vanish at some 
node $i$, $\widehat \sigma_{i,k}^0$ is not guaranteed to converge to $\sigma$. The random initialization of $\nu_{i,0}$ nevertheless activates the dominant modes almost
surely (a.s.), and hence at least one local
ratio $\widehat \sigma_{i,k}^0$ converges to $\sigma$ a.s., while the other locally observable spectral modulus is no larger than $\sigma$. We therefore propagate $\widehat \sigma_{i,k}^0$ using dynamic max consensus protocol in~\cite{Deplano_TAC_Dyn_max_consensus} where $m\in\mathbb N$ is the upper bound of the network diameter, as seen in \eqref{eq:sigma_max_consensus}.
Using the local estimate $\widehat{\bar h}_{i,k}$, which converges to $\bar h$ by \cref{ass:graph_connected}, agent $i$ constructs a time-varying estimation of curvature-dispersion signal $|{h_i}/{\widehat{\bar h}_{i,k}}-1|$.
To track the maximum of these $N$ signals at each iteration $k$, the dynamic max consensus protocol is utilized again as follows:
    \begin{equation}
    \widehat \eta_{i,k}^0=\left|\frac{h_i}{\widehat{\bar h}_{i,k}} -1\right|, \ \widehat\eta_{i,k}^\ell=\hspace{-1pt}\max_{j\in\mathcal N_i\cup \{i\}}\hspace{-1pt}\left\{
    \widehat\eta_{j,k-1}^{\ell-1}\right\} \forall\,\ell=1,2,\ldots,m.
    \end{equation}
    
\cref{alg:distributed_certified_stepsize} summarizes the proposed DGT with Certified Adaptive Stepsizes (CA-DGT).
Equipped with the preliminary lemmas, we establish the almost sure convergence of CA-DGT.

\begin{algorithm}[t]
\caption{\textbf{C}ertified \textbf{A}daptive \textbf{DGT} (CA-DGT).}
\label{alg:distributed_certified_stepsize}
\begin{algorithmic}[1]
\State \textbf{Input:} Upper bound $m\in\mathbb N$ on the network diameter.
\State \textbf{Initialize:}
$\widehat{\bar h}_{i,0}=h_i$,
$\nu_{i,0}\stackrel{\rm i.i.d.}{\sim}\mathcal D_\nu$,
$\widehat{\sigma}_{i,-1}^{\ell}=0$,
$\widehat{\eta}_{i,-1}^{\ell}=0$ for each
$i\in\mathcal V$ and $\ell=0,1,\ldots,m$, and
$\varepsilon_\nu>0$.
\State Each agent $i\in\mathcal V$ repeats the following steps from $k=0$:
\Statex \qquad
$\displaystyle
\widehat{\bar h}_{i,k+1}
\leftarrow
\sum\nolimits_{j\in\mathcal N_i\cup\{i\}}
w_{ij}\widehat{\bar h}_{j,k}$.
\Statex \qquad
$\displaystyle
\nu_{i,k+1}
\leftarrow
\sum\nolimits_{j\in\mathcal N_i\cup\{i\}}
w_{ij}\nu_{j,k}$.
\Statex \qquad
\textbf{If}
$k\geq2$ \textbf{and}
$|\nu_{i,k-1}-\nu_{i,k-2}|\geq\varepsilon_\nu$
\textbf{then}
\Statex \qquad\qquad
$\displaystyle
\widehat{\sigma}_{i,k}^{0}
\leftarrow
\sqrt{
     {|\nu_{i,k+1}-\nu_{i,k}|}/
     {|\nu_{i,k-1}-\nu_{i,k-2}|}
}$.
\Statex \qquad
\textbf{else}
\Statex \qquad\qquad
$\widehat{\sigma}_{i,k}^{0}
\leftarrow
\widehat{\sigma}_{i,k-1}^{0}$.
\Statex \qquad
$\displaystyle
\widehat{\sigma}_{i,k}^{\ell}
\leftarrow
\max_{j\in\mathcal N_i\cup\{i\}}
\left\{
\widehat{\sigma}_{j,k-1}^{\ell-1}
\right\},
\quad
\ell=1,2,\ldots,m$.
\Statex \qquad
$\displaystyle
\widehat{\eta}_{i,k}^{0}
\leftarrow
|{h_i}/{\widehat{\bar h}_{i,k}}-1|$.
\Statex \qquad
$\displaystyle
\widehat{\eta}_{i,k}^{\ell}
\leftarrow
\max_{j\in\mathcal N_i\cup\{i\}}
\left\{
\widehat{\eta}_{j,k-1}^{\ell-1}
\right\},
\quad
\ell=1,2,\ldots,m$.
\Statex \qquad
\textbf{If} $\widehat{\sigma}_{i,k}^{m}<1$ \textbf{then}
\Statex \qquad\qquad
$\displaystyle
\beta_{i,k}
\leftarrow
\widehat{\beta}^{\star}
\bigl(
\widehat{\sigma}_{i,k}^{m},
\widehat{\eta}_{i,k}^{m}
\bigr)$
by \eqref{eq:r_cert} and \eqref{eq:cert_beta_arg_min}.
\Statex \qquad
\textbf{else}
\Statex \qquad\qquad
$\beta_{i,k}\leftarrow0$.
\Statex \qquad
Perform DGT in \eqref{eq:dgt_scalar} with stepsize
$\displaystyle
\alpha_{i,k}:=
\beta_{i,k}/\widehat{\bar h}_{i,k}$.
\end{algorithmic}
\end{algorithm}

\begin{theorem}[Almost sure convergence of CA-DGT]
\label{thm:adaptive_certified_stepsize_convergence}
Let \cref{ass:graph_connected,ass:cost_function} hold. For fixed $\sigma\in[0,1)$ and $\eta\ge0$, the stepsizes generated by ~\cref{alg:distributed_certified_stepsize} satisfy $\beta_{i,k}\to\widehat \beta^\star(\sigma,\eta)$ a.s. for all $i\in\mathcal V$. Consequently, the CA-DGT algorithm has $\lim_{k\rightarrow \infty}x_{i,k}=x^\star$ a.s. for all $i\in\mathcal V$.
\end{theorem}

\begin{IEEEproof}
As discussed after~\eqref{eq:sigma_estimate}, 
$\widehat{\sigma}_{i,k}^{0}$ converges a.s. to limit whose
network maximum is $\sigma$. Moreover, \cref{ass:graph_connected}
and~\eqref{eq:ave_curv} yield
$\widehat{\bar h}_{i,k}\to\bar h$, and hence
$\widehat{\eta}_{i,k}^{0}
=|h_i/\widehat{\bar h}_{i,k}-1|
\to |h_i/\bar h-1|$, whose network maximum is $\eta$.
Applying \cite[Corollary~3]{Deplano_TAC_Dyn_max_consensus} gives
$\widehat{\sigma}_{i,k}^{m}\to\sigma$ a.s. and
$\widehat{\eta}_{i,k}^{m}\to\eta$ for all
$i\in\mathcal V$.

By the continuity of the certified stepsize rule from \cref{lem:cert_stepsize_unique_continuous}, it follows $\beta_{i,k}=\widehat\beta^\star(\hat\sigma_{i,k}^m,\hat\eta_{i,k}^m)\to\widehat\beta^\star(\sigma,\eta)$ $\forall i\in\mathcal V$ a.s. Equivalently, one has $\bm\beta_k:=\operatorname{diag}(\beta_{1,k},\ldots,\beta_{N,k})\to\widehat\beta^\star(\sigma,\eta)I$ a.s. Since the entries of $A(\bm\beta_k)$ are continuous in $\bm\beta_k$, with a slight abuse of notation, we have $A(\bm \beta_k)\to A(\widehat\beta^\star(\sigma,\eta))$ a.s. By \cref{thm:N_node_certificate} and the optimality of $\widehat\beta^\star(\sigma,\eta)$ in \eqref{eq:cert_beta_arg_min}, $\widehat\esr_{\rm DGT}^\star(\sigma,\eta)<1$.
Hence, \cref{lem:stability_certificate} implies that
$A(\widehat\beta^\star(\sigma,\eta))$ is Schur. For any $Q_{\rm L}\succ0$, there exists $P_{\rm L}\succ0$ such that $A(\widehat\beta^\star(\sigma,\eta))^\top P_{\rm L} A(\widehat\beta^\star(\sigma,\eta))-P_{\rm L}=-Q_{\rm L}$. Since $Q_{\rm L}\succ0$ and $P_{\rm L}\succ0$, there exists $c\in(0,1)$ such that $P_{\rm L}-Q_{\rm L}\preceq cP_{\rm L}$, and thus $A(\widehat\beta^\star(\sigma,\eta))^\top P_{\rm L} A(\widehat\beta^\star(\sigma,\eta))\preceq cP_{\rm L}$. By the continuity of $A\mapsto A^\top P_{\rm L} A$, for almost every sample path there exist a finite $K\in\mathbb N$ and a constant $\bar c\in(c,1)$ such that $A(\bm \beta_k)^\top P_{\rm L} A(\bm \beta_k)\preceq \bar cP_{\rm L}$ for all $k\ge K$. Therefore, with $V_k:=e_k^\top P_{\rm L} e_k$ for the reduced error state $e_k:=\operatorname{col}(\widehat y_{c,k},\widehat x_{d,k},\widehat y_{d,k})$, one has $V_{k+1}\le\bar c V_k$ for all $k\ge K$. It follows that $e_k$ converges exponentially to the origin. Consequently, $x_k\to \mathbf 1x^\star$ a.s., namely, $\lim_{k\to\infty}x_{i,k}=x^\star$ a.s. for all $i\in\mathcal V$.
\end{IEEEproof}

\cref{alg:distributed_certified_stepsize} implements an adaptive, distributed computation of the certified stepsize that solves~\eqref{eq:cert_beta_arg_min}. Each agent estimates global parameters $\sigma$ and $\eta$ locally, and \cref{thm:adaptive_certified_stepsize_convergence} ensures that each uncoordinated stepsize $\alpha_i$ converges a.s. to the centralized certified stepsize $\widehat\alpha^\star(\sigma,\eta)$.

\section{Simulations}
In this section, we numerically validate \cref{alg:distributed_certified_stepsize} and estimate how conservative the proposed stepsize is.

First, we consider $5$ agents on a ring graph with Metropolis weights and curvatures $\{h_i\}_{i\in\mathcal V}=\{4,7,8,9,10\}$. ${b_i}$ and ${c_i}$ are independently drawn from $\mathcal U(0,10)$. \cref{fig:hatbar_h,fig:hat_eta} report the local estimates of $\bar h$, $\sigma$, and $\eta$, respectively. For stepsize selection, each agent queries a precomputed lookup table of $\widehat\beta^\star(\widehat\sigma_{i,k}^m,\widehat\eta_{i,k}^{m})$ over $\widehat\sigma_{i,k}^m\in[0,1)$ and $\widehat\eta_{i,k}^{m}\in[0,\bar\eta]$, with $\bar\eta$ chosen sufficiently large. The uncoordinated stepsizes $\alpha_{i,k}$ are shown in \cref{fig:beta}. All estimates of the three global parameters rapidly converge to the respective true values and the stepsizes $\alpha_{i,k}$ settle to $\widehat\alpha^\star(\sigma,\eta)$ after approximately $k=10$ iterations.
\cref{fig:x_error_compare} compares the optimization error of CA-DGT with (i) the constant stepsize designed by~\cite[Theorem~1]{pu2021_Mathprogram_DGT} (label ``DGT'') and (ii) the local stepsizes updated according to~\cite[Theorem~1]{Adgt_arXiv_2025} (label ``AdGT'').
CA-DGT converges faster, consistently with the less conservative design in~\eqref{eq:cert_beta_arg_min} that approximates the optimal constant stepsize.

\begin{figure}[t]
    \centering

    \makebox[\columnwidth][c]{%
    \subfloat[]
    {
        \includegraphics[width=0.53\columnwidth,trim=8 0 8 3,clip]{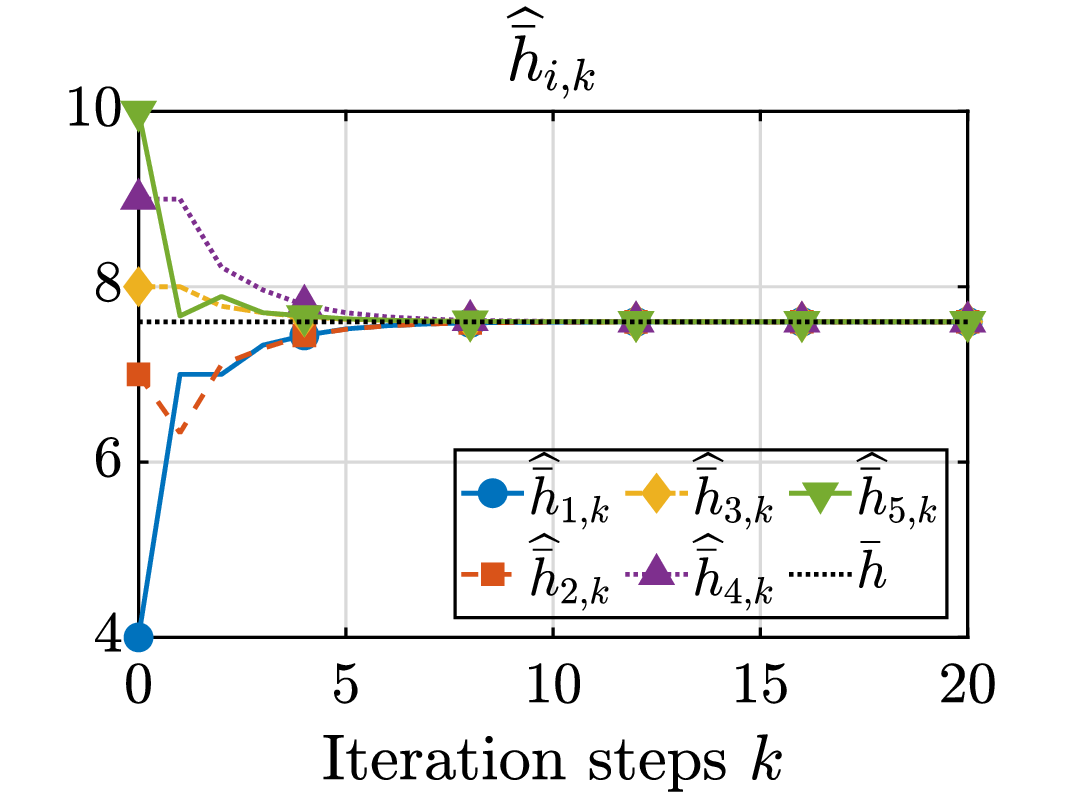}
        \label{fig:hatbar_h}
    }%
    \hspace{-0.055\columnwidth}%
    \subfloat[]
    {
        \includegraphics[width=0.53\columnwidth,trim=8 0 8 3,clip]{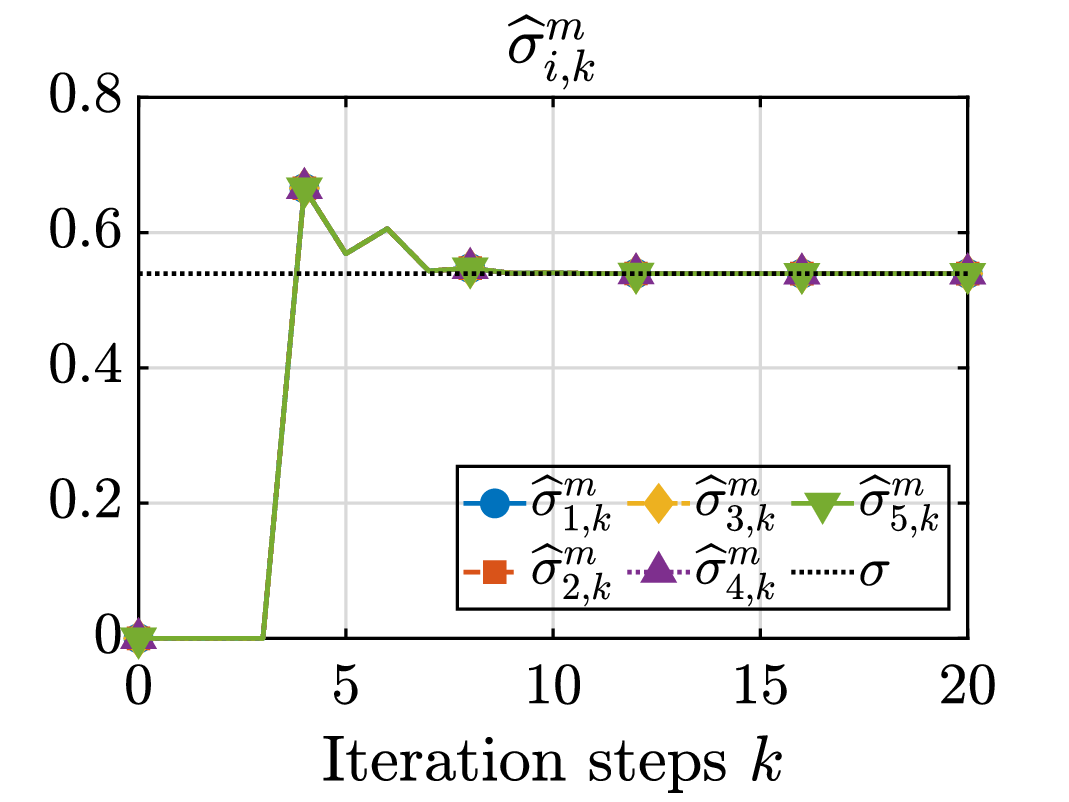}
        \label{fig:hat_sigma}
    }%
    }\\[-0.7em]

    \makebox[\columnwidth][c]{%
    \subfloat[]
    {
        \includegraphics[width=0.53\columnwidth,trim=8 0 8 3,clip]{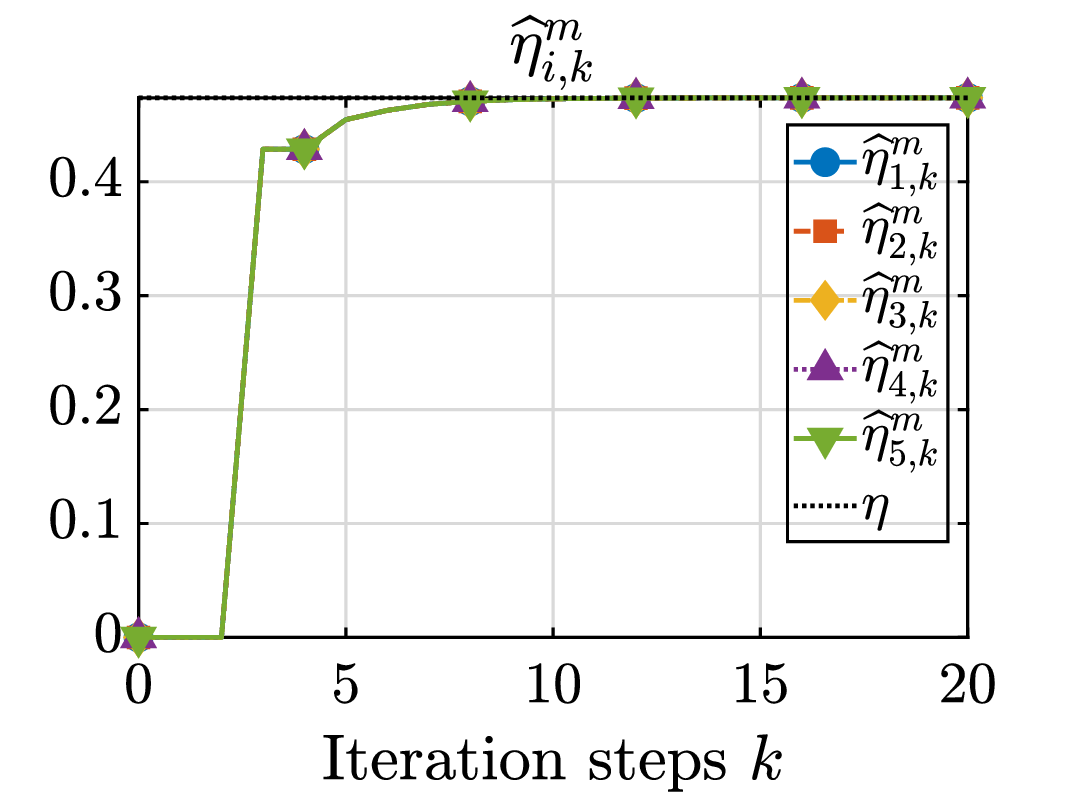}
        \label{fig:hat_eta}
    }%
    \hspace{-0.055\columnwidth}%
    \subfloat[]
    {
        \includegraphics[width=0.53\columnwidth,trim=8 0 8 3,clip]{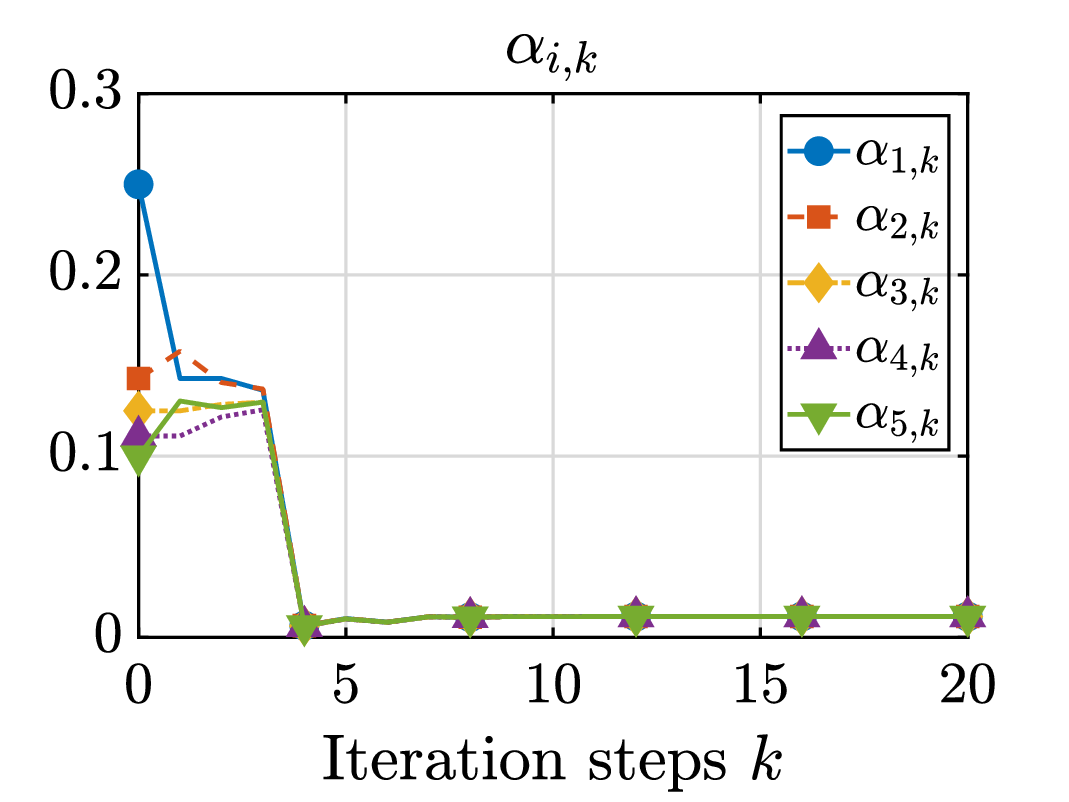}
        \label{fig:beta}
    }%
    }

    \vspace{-0.5em}
    \caption{Evolutions of the distributed estimates and adaptive stepsize in the proposed \cref{alg:distributed_certified_stepsize}. 
    (a) Local estimation of $\bar h$. 
    (b) Local estimation of $\sigma$. 
    (c) Local estimation of $\eta$. 
    (d) Adaptive stepsize $\alpha_{i,k}$.}
    \label{fig:adaptive_dgt_results}
\end{figure}

\begin{figure}[t]
    \centering
    \includegraphics[width=0.41\textwidth]{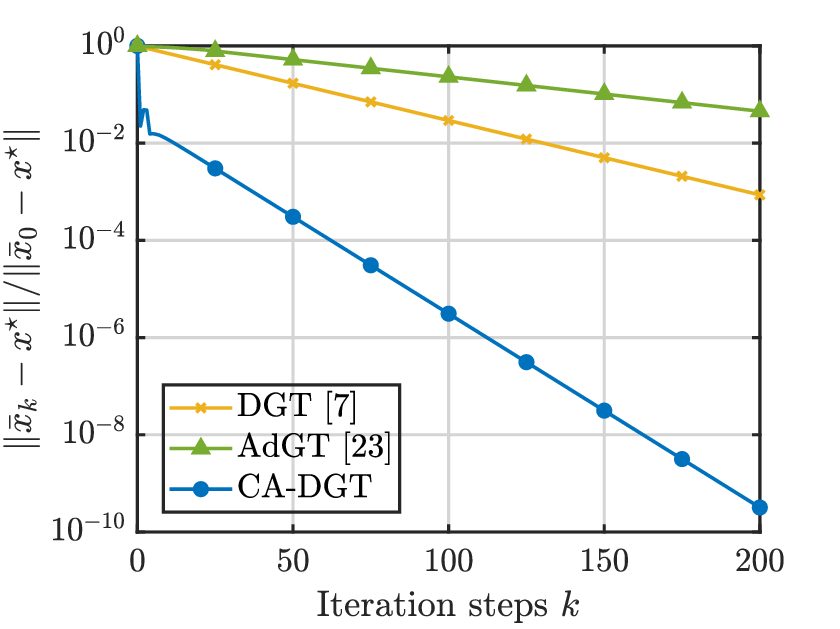}
    \vspace{-5pt}
    \caption{Comparison of the average normalized state error $|\bar x_k-x^\star|/|\bar x_0-x^\star|$ for DGT \cite{pu2021_Mathprogram_DGT}, AdGT \cite{Adgt_arXiv_2025} and CA-DGT where $\bar x_k=1/N \sum_{i=1}^Nx_{i,k}$.}
    \label{fig:x_error_compare}
    \vspace{-3pt}
\end{figure}

\begin{figure}[t]
\centering
\includegraphics[width=0.41\textwidth,trim={0 0pt 0 0pt},clip]{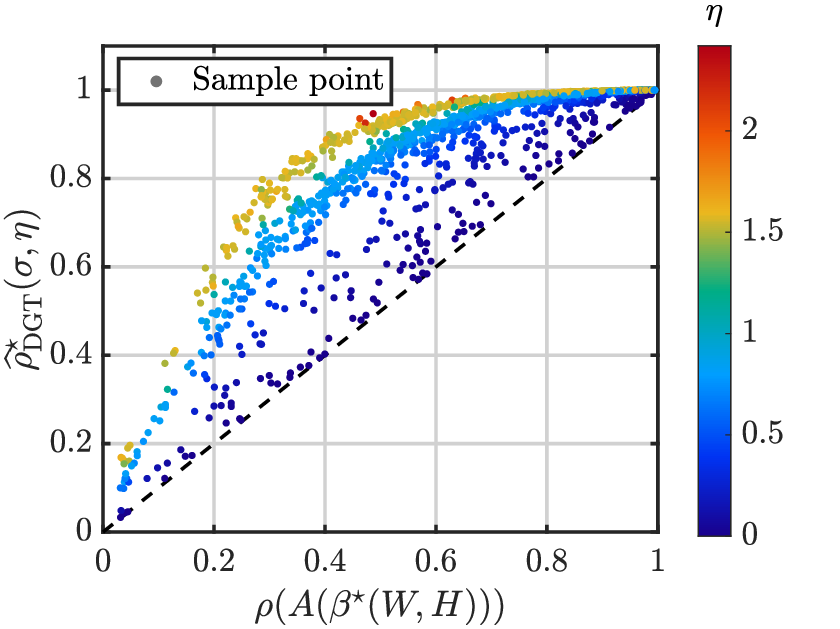} 
\vspace{-5pt}
\caption{Comparison between the minimal certified radius $\widehat\esr_{\rm DGT}^\star(\sigma,\eta)$ and the minimal convergence factor $\esr(A(\beta^\star(W,H)))$ for different  random graphs $W$ and curvatures $H$.}
\label{fig:rho_compare}
    \vspace{-3pt}
\end{figure}

\begin{figure}[t]
    \centering
    \includegraphics[width=0.41\textwidth]{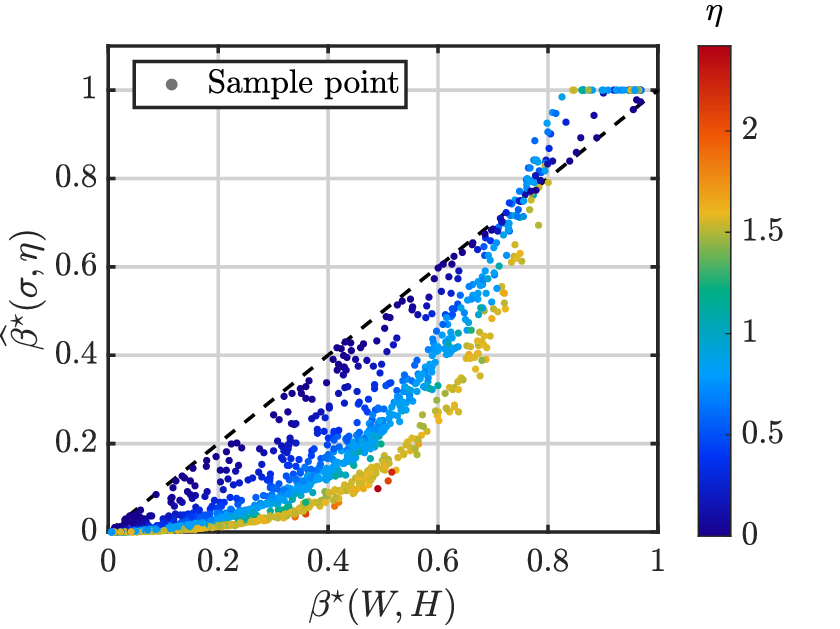}
    \vspace{-4pt}
    \caption{Comparison between the certified stepsize $\widehat\beta^\star(\sigma,\eta)$ and the optimal stepsize $\beta^\star(W,H)$ for different  random graphs $W$ and curvatures $H$.}
    \label{fig:beta_compare}
    \vspace{-3pt}
\end{figure}

The next experiment evaluates how conservative the certified stepsize $\widehat\beta^\star(\sigma,\eta)$ is compared to the optimal stepsize $\beta^\star(W,H)\in\arg\min_{\beta>0}\esr(A(\beta))$ for a given reduced DGT matrix $A(\beta)$. Since the latter problem is nonconvex, we numerically compute $\beta^\star(W,H)$ via a grid search over the same admissible range $0<\beta<2$.
We consider $N=2000$ agents and generate $2000$ random problem instances.
Each instance comprises one random network and local curvatures $\{h_i\}_{i\in\mathcal V}$ randomly generated within $[1,10]$.
\cref{fig:rho_compare} compares the minimal certified radius $\widehat\esr_{\rm DGT}^\star(\sigma,\eta)$ with the minimal convergence factor $\esr(A(\beta^\star(W,H)))=\esr(\alpha^\star;W,H)$.
Each point corresponds to one random problem instance $W,H$.
Although the points lie above the diagonal (recall that $\widehat\esr_{\rm DGT}(\beta,\sigma,\eta)$ upper bounds $\esr(A(\beta))$), many samples are relatively close to the diagonal for small $\eta$.
The same behavior is evident in \cref{fig:beta_compare} which compares the certified stepsize with numerically computed $\beta^\star(W,H)$.
This indicates that the certificate is reasonably tight under nearly uniform curvatures. Note that a cluster of sample points appears near $\widehat\beta^\star(\sigma,\eta)=1$.
These cases correspond to fast-mixing dense graphs (recall that $\beta^\star(W,H)=1$ for $\sigma=0$, see \cref{sec:complete_graph}), for which the certified bound $\widehat\beta^\star(\sigma,\eta)$ favors the largest admissible stepsize.

\section{Conclusion}\label{sec:conclusion}
We propose a principled stepsize design for fast convergence of DGT under scalar quadratic costs.
Our convergence analysis only uses curvature dispersion $\eta$ and essential spectral radius of the consensus matrix $\sigma$, in contrast to global Lipschitz constant, enabling analytical insight on how those quantities affect the convergence of DGT.
Building on the analysis, we propose a computationally tractable constant stepsize optimization which approximately minimizes the convergence factor of the DGT system matrix using only $\sigma$ and $\eta$. Finally, we develop a distributed implementation of adaptive uncoordinated stepsizes which converge to the optimized certified stepsize through local estimation of the two global parameters. 

Extensions beyond scalar quadratics are nontrivial even in a centralized framework \cite{adaptive_ifac}. For multi-dimensional quadratics, the present scalar analysis can be applied modewise when the local Hessians share common eigendirections. For general local Hessians, extending \cref{lem:stability_certificate} would naturally depend on perturbation analysis on matrix blocks, which may lead to a more conservative certified stepsize compared to that of the scalar quadratic case.
For smooth strongly convex objectives, Hessians vary along the iterations, leading to a time-varying robust analysis, which is more challenging. These extensions will be pursued in future work.

\appendices

\section{Proof of \texorpdfstring{\cref{lem:cert_stepsize_unique_continuous}}{Lemma~4}}\label{appendix:A}

We restrict the analysis to the range $\beta\in(0,1]$, because for every $\beta\in(0,1]$, its reflected stepsize $\beta'=2-\beta\in[1,2)$ yields no smaller certified radius: $\widehat\esr_{\rm DGT}(\beta,\sigma,\eta)\le \widehat\esr_{\rm DGT}(\beta',\sigma,\eta)$.

\subsubsection*{Uniqueness} The boundary cases follow from \cref{sec:special-regimes}; we have $\widehat\beta^\star(0,\eta)=1,\,\forall\,\eta\ge0$ and $\widehat\beta^\star(\sigma,0)=1-\sqrt{\sigma}, \,\forall\,\sigma\in[0,1)$.
Consider the general case $0<\sigma<1$, $\eta>0$. For given $r\in(0,1)$, define the set of stepsizes certified at radius $r$ as
$
    \mathcal B_r := \left\{ \beta\in(0,1]: \ R_0(\beta,\sigma)<r,\ \Phi(r,\beta,\sigma,\eta)\le1 \right\}$.
Let $\xi:=\beta-(1-r)\in(0,r]$. On this domain, together with the condition $R_T(1-r+\xi,\sigma)<r$, the inequality $\Phi(r,\beta,\sigma,\eta)\le1$ is equivalent to $\Theta_r(\xi)\le r$
where $\Theta_r(\xi) := R_T(1-r+\xi,\sigma) + \Theta_r^{\rm het}(\xi) \sqrt{\sigma(1+r) }$ with $\Theta_r^{\rm het}(\xi):=\sqrt{\eta(1-r+\xi) + \eta^2\frac{(1-r+\xi)^2}{\xi}}$. Then $\mathcal B_r$ is mapped to $\{\xi\in(0,r]:\Theta_r(\xi)\le r\}$.
Direct calculations show that $\Theta_r^{\rm het}(\xi)$ decreases in $(0,\xi_N)$ and increases in $(\xi_N,\infty)$ where $\xi_N=(1-r)\sqrt{\frac{\eta}{1+\eta}}$. On $(0,\min\{\xi_N,r\}]$, the positive curvature of $\Theta_r^{\rm het}(\xi)\sqrt{\sigma(1+r)}$ dominates the curvature of each branch of $R_T$. For $\xi\ge\xi_N$, both $\Theta_r^{\rm het}(\xi)$ and $R_T(1-r+\xi,\sigma)$ are non-decreasing. Thus $\Theta_r$ has a single-valley structure in $(0,r]$ and admits a unique global minimizer.
Define $ \chi(r):=\min_{\xi\in(0,r]}\Theta_r(\xi)$. Then $\mathcal B_r\neq\emptyset$ is equivalent to $\chi(r)\le r$.
By \cref{thm:N_node_certificate}, there exist $\bar{\beta}\in(0,1]$ and $\bar{r}\in(R_0(\bar\beta,\sigma),1)$ such that $\mathcal B_{\bar{r}}\neq\emptyset$ and $\chi(\bar r)\le \bar r$. 

For brevity, let $r^\star:=\inf\{r\in(0,1):\mathcal B_r\neq\emptyset\}$. Since $r^\star<1$ and $\eta>0$, $\Theta_r(\xi)\to\infty$ as $\xi\downarrow0$, uniformly for $r$ near $r^\star$. Hence the minimization defining $\chi(r):=\min_{\xi\in(0,r]}\Theta_r(\xi)$ can be restricted locally to a common compact subset of $(0,r]$, and Berge's maximum theorem gives the continuity of $\chi$ at $r^\star$. By the definition of $r^\star$, there exists $r_n\downarrow r^\star$ such that $\mathcal B_{r_n}\neq\emptyset$, equivalently $\chi(r_n)\le r_n$. Passing to the limit gives $\chi(r^\star)\le r^\star$. If the inequality was strict, continuity would give some $r<r^\star$ with $\chi(r)<r$, contradicting the definition of $r^\star$. Thus $\chi(r^\star)=r^\star$.
Let $\xi_{r^\star}^\star$ be the minimizer of $\Theta_{r^\star}$ in $(0,r^\star]$. The single-valley structure gives $\{\xi\in(0,r^\star]:\Theta_{r^\star}(\xi)\le r^\star\}=\{\xi_{r^\star}^\star\}$. Recalling that $\xi=\beta-(1-r)$, this implies $\mathcal B_{r^\star}=\{1-r^\star+\xi_{r^\star}^\star\}$. Hence every minimizer of $\widehat\esr_{\rm DGT}(\beta,\sigma,\eta)$ coincides with $1-r^\star+\xi_{r^\star}^\star$, proving uniqueness.

\subsubsection*{Continuity}
Consider $(\sigma,\eta)\to(0,\bar\eta)$ with $\bar\eta\ge0$. Take $\beta=1$ and choose $r_\sigma=\sigma^a$ with $a\in(0,1/3)$. Then, for small $\sigma>0$, one has $r_\sigma>R_0(1,\sigma)$ and $\Phi(r_\sigma,1,\sigma,\eta)\to0$. Hence $\widehat\esr^\star_{\rm DGT}(\sigma,\eta)\le r_\sigma\to0$. Since every feasible radius satisfies $r>R_0(\beta,\sigma)\ge1-\beta$, the minimizer satisfies $0\le1-\widehat\beta^\star(\sigma,\eta)\le\widehat\esr^\star_{\rm DGT}(\sigma,\eta)$. Therefore, $\widehat\beta^\star(\sigma,\eta)\to1=\widehat\beta^\star(0,\bar\eta)$.

It remains to prove when $\bar\sigma\in(0,1)$. We use the following fact: after expanding
$\widehat\esr_{\rm DGT}(0,\sigma,\eta)=1$, $(\beta,\sigma,\eta)\mapsto\widehat\esr_{\rm DGT}(\beta,\sigma,\eta)$ is jointly continuous
on $[0,1]\times(0,1)\times[0,\infty)$. This follows from the joint continuity of $\Phi$ and $R_0$, the strict monotonicity of $\Phi$ in $r$, and the convention
$\widehat\esr_{\rm DGT}(\beta,\sigma,\eta)=1$ at the empty-set boundary. For $\eta=0$, $\Phi(r,\beta,\sigma,0)\equiv0$ gives $\widehat\esr_{\rm DGT}(\beta,\sigma,0)=R_0(\beta,\sigma)$, while endpoint continuity follows from $R_0(\beta,\sigma)\le \widehat\esr_{\rm DGT}(\beta,\sigma,\eta)\le R_0(\beta,\sigma)+\varepsilon$ for all sufficiently small $\eta>0$.
Let $(\sigma_n,\eta_n)\to(\bar\sigma,\bar\eta)$ with $\bar\sigma\in(0,1)$ and $\bar\eta\ge0$, and set $\beta_n^\star:=\widehat\beta^\star(\sigma_n,\eta_n)$. Any subsequence admits a further one, still denoted by $\beta_n^\star$, with $\beta_n^\star\to\bar\beta$. By optimality, $\widehat\esr_{\rm DGT}(\beta_n^\star,\sigma_n,\eta_n)\le\widehat\esr_{\rm DGT}(\beta,\sigma_n,\eta_n), \forall\,\beta\in[0,1]$. By joint continuity, it has $\widehat\esr_{\rm DGT}(\bar\beta,\bar\sigma,\bar\eta)\le\widehat\esr_{\rm DGT}(\beta,\bar\sigma,\bar\eta), \forall\,\beta\in[0,1]$. Noting uniqueness gives $\bar\beta=\widehat\beta^\star(\bar\sigma,\bar\eta)$, hence every convergent subsequence has the same limit, proving $\widehat\beta^\star(\sigma_n,\eta_n)\to\widehat\beta^\star(\bar\sigma,\bar\eta)$.
Together with the boundary case at $\sigma=0$, $\widehat\beta^\star(\sigma,\eta)$ is continuous on $[0,1)\times[0,\infty)$.


\arxiv{}{%
\section{Proof of Proposition~\ref{prop:weak_heterogeneity_expansion}}\label{appendix:B}
\begin{IEEEproof}
      By Lemma~\ref{lem:homogeneous_envelope}, the optimal normalized stepsize minimizing the homogeneous root radius, i.e., the case of $\eta=0$, is $\beta_0=1-r_0$. At this point, the active branch in $R_T(\beta_0,\sigma)$ is $R_+(\beta_0,\sigma)=r_0$.
      
      Recall that $P_\lambda(z;\beta)=(z-\lambda)^2+\beta\lambda(z-1)$, and $R_+(\beta,\sigma)$ is implicitly defined by $P_\sigma(R_+(\beta;\sigma);\beta)=0$. Hence differentiating w.r.t. $\beta$ gives $\frac{\partial P_\sigma}{\partial z}\frac{\partial R_+}{\partial \beta} +\frac{\partial P_\sigma}{\partial \beta}=0$.
      Evaluating this relation at $z=r_0$, $\sigma=r_0^2$ and $\beta=\beta_0=1-r_0$, we obtain
      \[
     \kappa:=
     \left.\frac{\partial R_+}{\partial\beta}\right|_{\beta=\beta_0}
     =
     \frac{r_0}{r_0+2}.
     \]
     Denote $\delta_\beta=\beta_\eta-\beta_0$ and $\delta_r=r_\eta-r_0$. Since $R_+(\cdot,\cdot)$ is the active branch of $R_T(\cdot,\cdot)$ locally around $\beta_0$, according to Taylor expansion, we have 
     $$R_T(\beta_\eta,\sigma)=r_0+\kappa\delta_\beta+O(\delta_\beta^2).$$
    
     Let's define two gaps $g_1:=r_\eta-(1-\beta_\eta)$ and
     $g_2:=r_\eta-R_T(\beta_\eta,\sigma)$ where $r_\eta:=r_0+\delta_r$.
     At leading order, one has $g_1=\delta_r+\delta_\beta$ and
     $g_2=\delta_r-\kappa\delta_\beta$ which imply
     \begin{equation}\label{eq:delta_r_beta}
         \delta_r=\frac{g_2+\kappa g_1}{1+\kappa},\quad
         \delta_\beta=\frac{g_1-g_2}{1+\kappa}, \quad
         r_\eta=r_0+\frac{g_2+\kappa g_1}{1+\kappa}.
     \end{equation}
     One can observe that minimizing $r_\eta$ is
     equivalent, at leading order, to minimizing $g_2+\kappa g_1$.
    
     At the optimum, the certificate condition is active, i.e.,
     $\Phi(r_\eta,\beta_\eta,\sigma,\eta)=1$.
     Using $g_1=r_\eta-1+\beta_\eta$ and
     $g_2=r_\eta-R_T(\beta_\eta,\sigma)$, this equality gives $g_2^2=
     C_\eta\left(\eta+\frac{\beta_\eta\eta^2}{g_1}\right)$ where
     $C_\eta:=\beta_\eta\sigma(1+r_\eta)$.
     As $\eta\to0$, one has $\beta_\eta\to\beta_0$ and $r_\eta\to r_0$ which  result in $C_\eta\to C_0:=\beta_0\sigma(1+r_0)$.
     Hence the leading-order reduced objective is $\Psi_\eta(g_1):=\sqrt{C_0}\left(
     \eta+\frac{\beta_0\eta^2}{g_1}\right)^{1/2}
     +\kappa g_1$.
     The minimizer satisfies $g_1\gg\eta$ and $g_1\ll\eta^{1/2}$, because the first-order condition $\kappa=\frac{\sqrt{C_0}\beta_0\eta^2}
     {2g_1^2\left(\eta+\frac{\beta_0\eta^2}{g_1}\right)^{1/2}}
     \bigl(1+o(1)\bigr)$ rules out both $g_1=O(\eta)$ and $g_1=\Omega(\sqrt\eta)$.
     Therefore, we obtain $\left(\eta+\frac{\beta_0\eta^2}{g_1}
     \right)^{1/2}=\sqrt\eta+\frac{\beta_0}{2}\frac{\eta^{3/2}}{g_1}+
     o\left(\frac{\eta^{3/2}}{g_1}\right)$.
     After removing the constant term $\sqrt{C_0\eta}$, the first
     non-constant part of the objective is $\frac{\sqrt{C_0}\beta_0}{2}\frac{\eta^{3/2}}{g_1}+\kappa g_1$. 
     Minimizing this expression gives $g_1=\left(\frac{\sqrt{C_0}\beta_0}{2\kappa}\right)^{1/2}\eta^{3/4}+o(\eta^{3/4})$.
     Consequently, we have $g_2=\sqrt{C_0 \eta}+o(\sqrt\eta)$. 
     Substituting these estimates into the linear relations for
     $\delta_r$ and $\delta_\beta$ as in \eqref{eq:delta_r_beta} gives
     \[
     \delta_r=\frac{\sqrt{C_0}}{1+\kappa}\sqrt\eta+o(\sqrt\eta),
     \qquad
     \delta_\beta=-\frac{\sqrt{C_0}}{1+\kappa}\sqrt\eta+o(\sqrt\eta).
     \]
     By relations $r_\eta=r_0+\delta_r$ and
     $\beta_\eta=\beta_0+\delta_\beta$, the claimed expansions in \eqref{eq:beta_eta0} and \eqref{eq:r_eta0} follow.
\end{IEEEproof}}

\arxiv{}{%
\section{Proof of Proposition~\ref{prop:strong_hete_expansion}}\label{appendix:C}
\begin{IEEEproof}
     By the definition of $r_\eta=r_{\rm cert}(\beta_\eta,\sigma,\eta)$ and the continuity of $\Phi$, the infimum satisfies the closed certificate condition $\Phi(r_\eta,\beta_\eta,\sigma,\eta)\le1$. Inspecting this certificate condition, it requires $\beta_\eta\to0$ as $\eta\to+\infty$. Together with $1-\beta_\eta<r_\eta<1$, it implies $r_\eta\to1$. Moreover, one has $R_T(\beta_\eta,\sigma)\to\sigma$ according to \eqref{eq:R_T}. Therefore, we have the limit:
     \[
     \frac{\sigma(1+r_\eta)}{(r_\eta-R_T(\beta_\eta,\sigma))^2}\to\frac{2\sigma}{(1-\sigma)^2}.
     \]
    
     Recall the two gaps $g_1:=r_\eta-1+\beta_\eta$ and 
     $g_2:=r_\eta-R_T(\beta_\eta,\sigma)$. 
     In the strong-heterogeneity regime, $g_2\to1-\sigma>0$, and thus it only contributes the finite limiting factor above to the certificate. Therefore, the dominant order is determined by $g_1$.
     We write $\beta_\eta=c_\beta\eta^{-p}+o(\eta^{-p})$ and
     $g_1=c_g\eta^{-q}+o(\eta^{-q})$ where $c_\beta>0$ and $c_g>0$. The condition $0<g_1<\beta_\eta$ gives $q\ge p$. Moreover, the certificate condition requires $\beta_\eta\eta=O(1)$ and 
     $\frac{\beta_\eta^2\eta^2}{g_1}=O(1)$. Hence $p\ge1$ and $q\le 2p-2$. Combining these relations gives $p\le q\le 2p-2$ which holds only if $p\ge2$.
    
     Since $r_\eta=1-\beta_\eta+g_1$, the stability margin is $1-r_\eta=\beta_\eta-g_1$. To maximize this margin at the leading order, we choose the smallest feasible exponent, namely $p=2$. Then $p\le q\le 2p-2$ forces
     $q=2$, and hence
     \begin{equation}\label{eq:beta_eta_g1_expans_proof}
         \beta_\eta=c_\beta\eta^{-2}+o(\eta^{-2}),\quad g_1=c_g\eta^{-2}+o(\eta^{-2}).
     \end{equation}
    
     Since \(0<g_1<\beta_\eta\) is required for \(r_\eta<1\), the
     leading coefficients satisfy \(0<c_g<c_\beta\). Let
     \(d:=c_\beta-c_g>0\). Using
     \(r_\eta=1-\beta_\eta+g_1\), we obtain
     \begin{equation}\label{eq:r_eta_expans_proof}
         r_\eta=1-d\eta^{-2}+o(\eta^{-2}).
     \end{equation}
    
     It remains to determine $c_\beta$ and $d$. Since $\beta_\eta\eta=o(1)$,
     $\beta_\eta^2\eta^2/g_1=c_\beta^2/c_g+o(1)$, and $\frac{\sigma(1+r_\eta)}
     {(r_\eta-R_T(\beta_\eta,\sigma))^2}\to\frac{2\sigma}{(1-\sigma)^2}$,
     the leading-order certificate condition requires $\frac{2\sigma}{(1-\sigma)^2}\frac{c_\beta^2}{c_g}\le1$, or equivalently $d\le c_\beta-\frac{2\sigma}{(1-\sigma)^2}c_\beta^2$.
    
     Therefore, maximizing the leading-order stability margin is
     equivalent to maximizing
     $c_\beta-2\sigma c_\beta^2/(1-\sigma)^2$ over $c_\beta>0$. The maximizer and the corresponding maximum are
     \begin{equation}\label{eq:c^*_d^*}
         c_\beta^\star=\frac{(1-\sigma)^2}{4\sigma},
         \qquad
         d^\star=\frac{(1-\sigma)^2}{8\sigma}.
     \end{equation}
     Plugging \eqref{eq:c^*_d^*} into \eqref{eq:beta_eta_g1_expans_proof} and \eqref{eq:r_eta_expans_proof} ends the proof.
 \end{IEEEproof}}

\end{document}